\pdfoutput=1
\documentclass[12pt,a4paper]{amsart}

\calclayout

\usepackage[utf8]{inputenc}
    \usepackage[T1]{fontenc}
    \usepackage{libertine}
    \usepackage{stackengine,amsmath,amsthm,amssymb}
    \usepackage{mathdots}
    \usepackage{graphicx}
    \usepackage{tikz,tikz-cd}
            \tikzset{
                subseteq/.style={
                draw=none,
                edge node={node [sloped, allow upside down, auto=false]{$\subseteq$}}},
                Subseteq/.style={
                draw=none,
                every to/.append style={
                edge node={node [sloped, allow upside down, auto=false]{$\subseteq$}}}
                }
            }
            \tikzset{
                equal/.style={
                draw=none,
                edge node={node [sloped, allow upside down, auto=false]{$=$}}},
                Equal/.style={
                draw=none,
                every to/.append style={
                edge node={node [sloped, allow upside down, auto=false]{$=$}}}
                }
            }
            \tikzset{
            rotated_label/.style={anchor=south, rotate=90, inner sep=.5mm}
            }
    \usetikzlibrary{decorations}
    \usepackage{verbatim}
    \usepackage{booktabs}
    \usepackage{array}
    \usepackage{tabularx}
    \usepackage{booktabs}
    \usepackage{siunitx}
    \usepackage{rotating}
    \usepackage{adjustbox}
    \usepackage{longtable}
    \usepackage{caption} 
    \usepackage{stmaryrd}
    \usepackage{nicefrac}
    \usepackage{enumitem}
    \usepackage{xfrac}
    \usepackage[textwidth=0.8in,textsize=tiny]{todonotes}
    \usepackage{multirow}
    
\usepackage[backref=true,giveninits=true,doi=false,url=false,isbn=false,backend=biber]{biblatex}
    \newbibmacro{string+doiurlisbn}[1]{%
      \iffieldundef{doi}{%
        \iffieldundef{url}{%
          \iffieldundef{isbn}{%
            \iffieldundef{issn}{%
              #1%
            }{%
              \href{http://books.google.com/books?vid=ISSN\thefield{issn}}{#1}%
            }%
          }{%
            \href{http://books.google.com/books?vid=ISBN\thefield{isbn}}{#1}%
          }%
        }{%
          \href{\thefield{url}}{#1}%
        }%
      }{%
        \href{http://dx.doi.org/\thefield{doi}}{#1}%
      }%
    }
    \DeclareFieldFormat{title}{\usebibmacro{string+doiurlisbn}{\mkbibemph{#1}}}
    \DeclareFieldFormat[article,incollection,thesis,misc,inproceedings,online,inbook]{title}{\usebibmacro{string+doiurlisbn}{\mkbibquote{#1}}}
    \DeclareFieldFormat{year}{\usebibmacro{year-parenthesis}{\mkbibemph{#1}}}
\renewbibmacro*{date}{%
\iffieldundef{year}
{}
{\ifentrytype{misc}{\printtext[parens]{\printdate}}{\printdate}}}%
\usepackage{hyperref}
    \hypersetup{colorlinks=true,allcolors=blue}
    \usepackage[capitalize,nameinlink,noabbrev]{cleveref}

\usepackage{amsthm}
    \theoremstyle{plain}
        \newtheorem{theorem}{Theorem}[section]
        \newtheorem{lemma}[theorem]{Lemma}
        \newtheorem{proposition}[theorem]{Proposition}
        \newtheorem{corollary}[theorem]{Corollary}
        \newtheorem{conjecture}[theorem]{Conjecture}

    \theoremstyle{definition}
        \newtheorem{definition}[theorem]{Definition}

    \theoremstyle{remark}

        \newtheorem{remark}[theorem]{Remark}
        \newtheorem{question}[theorem]{Question}

\DeclareMathOperator{\Gal}{Gal}

\DeclareMathOperator{\disc}{disc}
\DeclareMathOperator{\car}{char}
\DeclareMathOperator{\End}{End}
\DeclareMathOperator{\Aut}{Aut}

\DeclareMathOperator{\Spec}{Spec}
\DeclareMathOperator{\Frac}{Frac}
\DeclareMathOperator{\trd}{trd}
\DeclareMathOperator{\nrd}{nrd}
\DeclareMathOperator{\Iso}{Iso}

\newcommand{\F}{\mathbb{F}}

\newcommand{\Z}{\mathbb{Z}}
\newcommand{\Q}{\mathbb{Q}}

\newcommand{\Ogotic}{\mathcal{O}}

\renewcommand{\mod}{\operatorname{mod}\ }

\newcolumntype{C}[1]{>{\centering\arraybackslash}p{#1}}

\newcommand\restr[2]{{ \left.\kern-\nulldelimiterspace #1 \vphantom{\big|} \right|_{#2} }}

\makeatletter
\@namedef{subjclassname@2020}{\textup{2020} Mathematics Subject Classification}
\makeatother

\title[Effective bounds on differences of singular moduli that are $S$-units]{Effective bounds on differences of singular moduli that are $S$-units}
    \author{Francesco Campagna}
    \address{Francesco Campagna - Max Planck Institute for Mathematics, Vivatsgasse 7, 53111 Bonn, Germany}
    \email{\href{mailto:campagna@mpim-bonn.mpg.de}{campagna@mpim-bonn.mpg.de}}
    \date{\today}
    
\usepackage[libertine,cmintegrals,cmbraces]{newtxmath}

\begin{document}
    \maketitle
    \begin{abstract}
        Given a singular modulus $j_0$ and a set of rational primes $S$, we study the problem of effectively determining the set of singular moduli $j$ such that $j-j_0$ is an $S$-unit. For every $j_0 \neq 0$, we provide an effective way of finding this set for infinitely many choices of $S$. The same is true if $j_0=0$ and we assume the Generalized Riemann Hypothesis. Certain numerical experiments will also lead to the formulation of a "uniformity conjecture" for singular $S$-units.
    \end{abstract}


\section{Introduction}

The present manuscript is devoted to the study of some diophantine properties of $j$-invariants of elliptic curves with complex multiplication defined over $\mathbb{C}$. These numbers, which are classically known by the name of \textit{singular moduli}, have been studied since the time of Kronecker and Weber, who were interested in explicit generation of class fields relative to imaginary quadratic fields \cite{Frei_1989}. In this respect, singular moduli prove to be a useful tool, since they are indeed algebraic integers which can be used to generate ring class fields of imaginary quadratic fields \cite[Theorem 11.1]{Cox_book_2013}.

During the last decade, there has been an increasing interest in understanding more diophantine properties of these invariants. One of the questions that, for instance, has been addressed is the following: given a set $S$ of rational primes, is the set of singular moduli that are $S$-units (\textit{singular $S$-units}) finite? In case of an affirmative answer, is it possible to provide an effective method to explicitly compute this set? This question, which has been originally motivated by the proof of some effective results of André-Oort type (see \cite{Bilu_Masser_Zannier_2013} and \cite{Kuhne_2012}), does not have at present a complete answer. Several partial results have nonetheless been achieved.

In \cite{BHK_2018} it is proved, building on the previous ineffective result of Habegger \cite{Habegger_2015}, that no singular modulus can be a unit in the ring of algebraic integers. This settles the case $S=\emptyset$ of the question. With different techniques, Li generalizes this theorem and proves in \cite{Li_2018} that for every pair $j_1,j_2 \in \overline{\mathbb{Q}}$ of singular moduli, the algebraic integer $\Phi_N(j_1,j_2)$ can never be a unit. Here $\Phi_N(X,Y) \in \mathbb{Z}[X,Y]$ denotes the classical modular polynomial of level $N$, so we recover the main result of \cite{BHK_2018} by setting $j_2=0$ and $N=1$. In a different direction, the fact that no singular modulus is a unit has been used by the author of this manuscript to prove that, if $S_0$ is the infinite set of primes congruent to $1$ mod $3$, then the set of singular moduli that are $S_0$-units is empty \cite{Campagna_2020}. Moreover, very recently Herrero, Menares and Rivera-Letelier gave an ineffective proof of the fact that for every fixed singular modulus $j_0 \in \overline{\mathbb{Q}}$ and for every finite set of primes $S$, the set of singular moduli $j$ such that $j-j_0$ is an $S$-unit is finite, see \cite{Herrero_Menares_Rivera_2020}  \cite{Herrero_Menares_Rivera_Part2} and \cite{Herrero_Menares_Rivera_Part3}.

In this paper we explore the possibility of providing, for a given singular modulus $j_0$ and for specific sets of primes $S$, an effective procedure to determine the set of all singular moduli $j$ such that $j-j_0$ is an $S$-unit. In order to better state our main results, we introduce some notation. First of all, we say that a singular modulus has discriminant $\Delta \in \mathbb{Z}$ if it is the $j$-invariant of an elliptic curve $E_{/\mathbb{C}}$ with complex multiplication by an order of discriminant $\Delta$. Let $j \in \overline{\mathbb{Q}}$ be a singular modulus of discriminant $\Delta$ and let $S \subseteq \mathbb{N}$ be a finite set of prime numbers. We call the pair $(j,S)$ a \textit{nice $\Delta$-pair} if the following two conditions hold:
\begin{enumerate}
    \item every prime $\ell \in S$ splits completely in $\mathbb{Q}(j)$;
    \item we have $\ell \nmid N_{\mathbb{Q}(j)/\mathbb{Q}}(j) N_{\mathbb{Q}(j)/\mathbb{Q}}(j-1728) \Delta$ for all $\ell \in S$, where $N_{\mathbb{Q}(j)/\mathbb{Q}}(\cdot)$ denotes the norm map from $\mathbb{Q}(j)$ to $\mathbb{Q}$.
\end{enumerate}
The first main result of the paper is the following.

\begin{theorem} \label{main theorem j-j0}
Let $(j_0,S)$ be a nice $\Delta_0$-pair with $\Delta_0 < -4$ and $\#S \leq 2$. Then there exists an effectively computable bound $B=B(j_0,S) \in \mathbb{R}_{\geq 0}$ such that the discriminant $\Delta$ of every singular modulus $j \in \overline{\mathbb{Q}}$ for which $j-j_0$ is an $S$-unit satisfies $|\Delta| \leq B$. Moreover, if the extension $\mathbb{Q} \subseteq \mathbb{Q}(j_0)$ is not Galois, then the discriminant $\Delta$ of any singular modulus $j$ such that $j-j_0$ is an $S$-unit is of the form $\Delta=p^{2n} \Delta_0$ for some prime $p \in S$ and some non-negative integer $n$.
\end{theorem}
The bound $B(j_0,S)$ in the statement of Theorem \ref{main theorem j-j0} can be made explicit from its proof. To give an idea of what kind of bounds one can get, we take $j_0=-3375$, the $j$-invariant of any elliptic curve with complex multiplication by $\mathbb{Z}[(1+\sqrt{-7})/2]$, and choose $S$ to be any subset of at most two elements in $\{13,17,19\}$. We get the following result.
\begin{theorem} \label{main theorem j+3375}
Let $j \in \overline{\mathbb{Q}}$ be a singular modulus of discriminant $\Delta$, and let $S:=\{13,17\}$. If $j+3375$ is an $S$-unit, then $|\Delta| \leq 10^{81}$. The same holds with $S'=\{13,19 \}$ and $S''=\{17,19 \}$.
\end{theorem}

In general, in order to construct nice $\Delta$-pairs it suffices to fix a singular modulus $j$ of discriminant $\Delta$ and to choose, among the set of primes splitting completely in $\mathbb{Q} \subseteq \mathbb{Q}(j)$, a finite subset $S$ satisfying condition (2) above. Since the set of rational primes that are totally split in $\mathbb{Q}(j)$ is infinite by the Chebotarëv's density theorem, this gives rise to infinitely many nice $\Delta$-pairs for a fixed discriminant $\Delta$. We remark that if $\mathbb{Q} \subseteq \mathbb{Q}(j)$ is not Galois, then every prime splitting completely in this extension will be also totally split in $\mathbb{Q}(\sqrt{\Delta})$ (see the end of the proof of Theorem \ref{main theorem j-j0}). Hence, in some cases one could use \cite[Theorem 2.2.1]{Campagna_PhD} to show that, for appropriate nice $\Delta_0$-pair $(j_0,S)$ with $\mathbb{Q} \subseteq \mathbb{Q}(j_0)$ non-Galois, the set of singular moduli $j \in \overline{\mathbb{Q}}$ for which $j-j_0$ is an $S$-unit is in fact empty. We point out that the set of singular moduli $j$ that generate a Galois extension of $\mathbb{Q}$ is finite, see Proposition \ref{prop:Galois_for_finitely_many_j}.

The reason why Theorem \ref{main theorem j-j0} only deals with sets $S$ containing at most two primes will be apparent from its proof, which we now sketch. Our strategy follows the same idea used in \cite{BHK_2018}: given a singular modulus $j \in \overline{\mathbb{Q}}$ such that $j-j_0$ is an $S$-unit, we compute the (logarithmic) Weil height $h(j-j_0)$. This is defined, for every $x \in \overline{\mathbb{Q}}$, as
\[
h(x)=\frac{1}{[K:\mathbb{Q}]} \sum_{v \in \mathcal{M}_K} [K_v:\mathbb{Q}_v] \log^+ |x|_v
\] where $K:=\mathbb{Q}(x)$ is the field generated by $x$ over the rationals, $\mathcal{M}_K$ is the set of all places of $K$, the integer $[K_v:\mathbb{Q}_v]$ is the local degree at the place $v$ and $\log^+|x|_v:=\log \max \{1,|x|_v\}$. Here, for every non-archimedean place $v$ corresponding to the prime ideal $\mathfrak{p}_v$ lying above the rational prime $p_v$, the absolute value $|\cdot |_v$ is normalized in such a way that 
\[
|x|_v=p_v^{- v_{\mathfrak{p}_v}(x)/e_v }
\]
where $e_v$ is the ramification index of $\mathfrak{p}_v$ over $p_v$ and $v_{\mathfrak{p}_v}(x)$ is the exponent with which $\mathfrak{p}_v$ appears in the prime ideal factorization of the $\mathcal{O}_K$-fractional ideal generated by $x$. Hence the logarithmic Weil height naturally decomposes into an "archimedean" and "non-archimedean" part.

Since $j-j_0$ is an algebraic integer, the non-archimedean part of its Weil height vanishes. In order to exploit the fact that the above difference is an $S$-unit, we rather compute the height of $(j-j_0)^{-1}$. Using standard properties of the Weil height, we obtain
\[
h(j-j_0)=h((j-j_0)^{-1})= \left( \text{archimedean part} \right) + \left( \text{non-archimedean part} \right)
\]
with
\[
\left( \text{non-archimedean part} \right)=\frac{1}{[\mathbb{Q}(j-j_0): \mathbb{Q}]} \sum_{\mathfrak{p}} f_\mathfrak{p} \cdot v_{\mathfrak{p}}(j-j_0) \log \ell_\mathfrak{p}
\]
where the sum is taken over the prime ideals of $\mathbb{Q}(j-j_0)$ lying above the rational primes contained in $S$ and, for every such prime $\mathfrak{p}$, we denote by $f_\mathfrak{p}$ and $\ell_\mathfrak{p}$ respectively the inertia degree and the residue characteristic of $\mathfrak{p}$. Our goal is to effectively bound this height from above and from below in such a way that the two bounds contradict each other when the absolute value of the discriminant of the singular modulus $j$ becomes large. This will give the desired effective bound. 

An upper bound for the archimedean part has been already studied in \cite{BHK_2018} and \cite{Cai_2020}. In order to estimate from above the non-archimedean part, we have to understand the valuation of $j-j_0$ at primes above $S$. This requires the use of some deformation-theoretic arguments involving quaternion algebras, and constitutes the technical core of the paper. We detail this discussion in \cref{Section 4}, which culminates in the proof of Theorem \ref{Theorem valuation j-j0}, where we obtain the seeked estimates. Concerning the lower bound for the Weil height, we compare it to the stable Faltings height of the elliptic curve with complex multiplication having $j$ as singular invariant. Using work of Colmez \cite{Colmez_1998} and Nakkajima-Taguchi \cite{Nakkajima_Taguchi_1991} it is possible to relate this Faltings height to the logarithmic derivative of the $L$-function corresponding to the CM field evaluated in $1$. The known lower bounds on this logarithmic derivative become strong enough for our purposes only if we restrict to sets $S$ containing no more than two primes.

When $\Delta_0 \in \{-3,-4\}$, \textit{i.e.} when $j_0 \in \{0,1728\}$, the same techniques also lead to similar finiteness results, but one has to be more careful in theses cases since the complex elliptic curves having $j_0$ as singular invariant possess non-trivial automorphisms. This is indeed a problem, and will force us to resort to the Generalized Riemann Hypothesis (GRH) in the case $j_0=0$. Here are the results that we obtain in these two cases.

\begin{theorem} \label{main theorem j-1728}
Let $S_0$ be the set of rational primes congruent to $1$ modulo $4$, let $\ell \geq 5$ be an arbitrary prime and set $S_\ell:=S_0 \cup \{\ell \}$. Then there exists an effectively computable bound $B=B(\ell) \in \mathbb{R}_{\geq 0}$ such that the discriminant $\Delta$ of every singular modulus $j \in \overline{\mathbb{Q}}$ for which $j-1728$ is an $S_\ell$-unit satisfies $|\Delta| \leq B$.
\end{theorem}

\begin{theorem} \label{main theorem j-0}
Let $S_0$ be the set of rational primes congruent to $1$ modulo $3$, let $\ell \geq 5$ be an arbitrary prime and set $S_\ell:=S_0 \cup \{\ell \}$. If the Generalized Riemann Hypothesis holds for the Dirichlet $L$-functions attached to imaginary quadratic number fields, then there exists an effectively computable bound $B=B(\ell) \in \mathbb{R}_{\geq 0}$ such that the discriminant $\Delta$ of every singular $S_\ell$-unit $j \in \overline{\mathbb{Q}}$ satisfies $|\Delta| \leq B$.
\end{theorem}

The statement of Theorem \ref{main theorem j-0} has been simplified for the sake of exposition in this introduction. Indeed, one does not need the full strength of GRH to carry out the proof, but only a weaker, more technical assumption on the logarithmic derivative at $s=1$ of the Dirichlet $L$-functions of imaginary quadratic fields. We refer the reader to Theorem \ref{true main theorem j-0} for the stronger result that we are actually going to prove.

After performing some numerical computations, one soon realizes that, given a singular modulus $j_0$ and a finite set of primes $S$, the upper bound for the number of singular moduli $j$ such that $j-j_0$ is an $S$-unit seems not to depend on the primes contained in $S$ but only on the size of the set $S$ itself. Since being an $S$-unit is a Galois-invariant property, this would entail a bound, depending only on $\#S$, on the size of the Galois orbits of such $j$'s and, by the Brauer-Siegel theorem \cite[Chapter XIII, Theorem 4]{Lang_book_NT_1994}, an analogous bound on their discriminants. Choosing $j_0=0$, this observation leads to the formulation of the following conjecture for singular $S$-units.

\begin{conjecture}
For every $s \in \mathbb{N}$, the number of singular moduli that are $S$-units for some set of rational primes $S$ with $\#S=s$ is finite.
\end{conjecture}

This conjecture, which we will call "uniformity conjecture for singular $S$-units", will be discussed in \cref{sec:uniformity conjecture}, where we also provide some numerical data to support it.

The manuscript is structured as follows. In \cref{sec: prelude} we recall known facts from the theory of complex multiplication and quaternion algebras, and we fix the terminology which will be used in the paper. In \cref{Section 4} we prove Theorem \ref{Theorem valuation j-j0}, which allows to bound the $\ell$-adic absolute value of differences of singular moduli for certain primes $\ell$. In \cref{sec: proof of main theorem j-j0} we provide a proof of Theorems \ref{main theorem j-j0} and \ref{main theorem j+3375} while in \cref{sec: proof the other two theorems} we give a proof of Theorems \ref{main theorem j-1728} and \ref{main theorem j-0}. \cref{sec: remove GRH?} discusses the optimality of the bounds found in Theorem \ref{Theorem valuation j-j0} in the case $j_0=0$. Finally in \cref{sec:uniformity conjecture} we provide numerical evidence for some uniformity conjectures concerning differences of singular moduli that are $S$-units.

\section{Prelude: CM elliptic curves, quaternion algebras and optimal embeddings} \label{sec: prelude}

We recall in this section some of the main definitions and results that will be used in the rest of the paper. We fix once and for all an algebraic closure $\overline{\mathbb{Q}} \supseteq \mathbb{Q}$ of the rationals.

A \textit{singular modulus} is the $j$-invariant of an elliptic curve defined over $\overline{\mathbb{Q}}$ with complex multiplication. For every imaginary quadratic order $\mathcal{O}$ of discriminant $\Delta \in \mathbb{Z}_{<0}$ there are exactly $C_{\Delta}$ isomorphism classes of elliptic curves over $\overline{\mathbb{Q}}$ with complex multiplication by $\mathcal{O}$, where $C_{\Delta} \in \mathbb{N}$ denotes the class number of the order $\mathcal{O}$. Hence, there are $C_{\Delta}$ corresponding singular moduli, which are all algebraic integers and form a full Galois orbit over $\mathbb{Q}$ (see \cite[Corollary 10.20]{Cox_book_2013}, \cite[Theorem 11.1]{Cox_book_2013} and \cite[Proposition 13.2]{Cox_book_2013}). We call them \textit{singular moduli of discriminant $\Delta$} or \textit{singular moduli relative to the order $\mathcal{O}$}. Reversing subject and complements, we will sometimes also speak of discriminant, CM order, CM field, etc... associated to a singular modulus $j$. 

Recall that, given a number field $K \subseteq \overline{\mathbb{Q}}$ and a set $S\subseteq \mathbb{N}$ of rational primes, an element $x \in K$ is called an \textit{$S$-unit} if for every prime $\mathfrak{p} \subseteq K$ not lying above any prime $p \in S$, we have $x \in \mathcal{O}_{K_\mathfrak{p}}^{\times}$, where $\mathcal{O}_{K_\mathfrak{p}} \subseteq K_\mathfrak{p}$ denotes the ring of integers in the completion $K_\mathfrak{p}$ of the number field $K$ at the prime $\mathfrak{p}$. Note that this definition does not depend on the particular number field $K$ containing $x$. Moreover, if $x$ is actually an algebraic integer, then $x$ is an $S$-unit if and only if its absolute norm $N_{K/\mathbb{Q}}(x)$ is divided only by primes in $S$. In this paper we are interested in the study of $S$-units of the form $j-j_0$ with $j,j_0 \in \overline{\mathbb{Q}}$ singular moduli. If $j_0=0$ is the unique singular modulus of discriminant $\Delta_0=-3$, we speak of \textit{singular $S$-units}. As we will see, the study of these  \textit{singular differences} is intimately related to the theory of supersingular elliptic curves and quaternion algebras. We summarize some relevant results from this theory. 

Let $k$ be a field of characteristic $\car(k)=\ell >0$ with algebraic closure $\overline{k} \supseteq k$ and let $E/k$ be an elliptic curve. We say that $E$ is \textit{supersingular} if $E[\ell](\overline{k})=\{O\}$ \textit{i.e.} if the unique $\ell$-torsion point of $E$ defined over $\overline{k}$ is the identity $O \in E(\overline{k})$. If this is the case, then the endomorphism ring $\End_{\overline{k}}(E)$ is isomorphic to a maximal order in the unique (up to isomorphism) quaternion algebra over $\mathbb{Q}$ ramified only at $\ell$ and $\infty$ (see \cite{Deuring_1941} or \cite[Proposition 42.1.7 and Theorem 42.1.9]{Voight_book_2020} for a modern exposition). If $k$ is a finite field, then by Deuring's lifting theorem \cite[Chapter 13, Theorem 14]{Lang_1987} every supersingular elliptic curve over $\overline{k}$ arises as the reduction of some elliptic curve with complex multiplication defined over a number field. Finding such a CM elliptic curve is difficult in general. In contrast, it is very easy to see for which primes a CM elliptic curve defined over a number field has good supersingular reduction. Namely, let $F$ be a number field with ring of integers $\mathcal{O}_F$ and let $E_{/F}$ be an elliptic curve with CM by an order in an imaginary quadratic field $K$. Fix a prime ideal $\mu \subseteq \mathcal{O}_F$ lying above a rational prime $\ell \in \mathbb{Z}$ that does not split in $K$. Since CM elliptic curves have potential good reduction everywhere (see \cite[VII, Proposition 5.5]{Silverman_2009}) we can assume, possibly after enlarging the field of definition $F$, that $E$ has good reduction at $\mu$ and that all the geometric endomorphisms of $E$ are defined over $F$. Then the reduced elliptic curve $\widetilde{E}:= E \text{ mod } \mu$ is supersingular by \cite[Chapter 13, Theorem 12]{Lang_1987}. Moreover, the natural reduction map modulo $\mu$ induces an injective ring homomorphism 
\[
\varphi: \End_{F}(E) \hookrightarrow \End_{\overline{\mathbb{F}}_\ell}(\widetilde{E})
\]
between the corresponding endomorphism rings (see \cite[II, Proposition 4.4]{Silverman_book_1994}). As we will see in Theorem \ref{deformation}, in many cases (depending on the prime $\ell$ and on the CM order of $E$) the above embedding will be \textit{optimal}, in the following sense. 

Let $\mathbb{B}$ be a quaternion algebra over $\mathbb{Q}$ and let $R \subseteq \mathbb{B}$ be an order, \textit{i.e.} a full $\mathbb{Z}$-lattice which is also a subring of $\mathbb{B}$. Let $\mathbb{Q} \subseteq K$ be a quadratic field extension and let $\mathcal{O} \subseteq K$ also be an order. Any ring homomorphism $\varphi: \mathcal{O} \to R$ can be naturally extended, after tensoring with $\mathbb{Q}$, to a ring homomorphism $K \to \mathbb{B}$ that we still denote by $\varphi$, with abuse of notation. We say that an injective ring homomorphism $\iota: \mathcal{O} \hookrightarrow R$ is an \textit{optimal embedding} if 
\[
\iota(K) \cap R= \iota(\mathcal{O})
\]
where the above intersection takes place in $\mathbb{B}$. There is a simple criterion which allows to determine whether a given imaginary quadratic order optimally embeds into a quaternionic order. In order to state it, let us denote by $\trd, \nrd: \mathbb{B} \to \mathbb{Q}$ respectively the reduced trace and the reduced norm in the quaternion algebra $\mathbb{B}$, see \cite[Section 3.3]{Voight_book_2020}. This notation will be in force for the rest of the paper.

\begin{lemma} \label{Gross Lattice}
Let $R$ be an order in a quaternion algebra $\mathbb{B}$ and $\mathcal{O}$ an order of discriminant $\Delta$ in an imaginary quadratic field $K$. Let $V\subseteq \mathbb{B}$ be the subspace of pure quaternions 
\[
V:=\{x\in \mathbb{B}: \trd(x)=0\}.
\]
Then $\mathcal{O}$ embeds (resp. optimally embeds) in $R$ if and only if $|\Delta|$ is represented (resp. primitively represented) by the ternary quadratic lattice
\[
R_0:= V \cap (\mathbb{Z}+2R)
\]
endowed with the natural scalar product induced by the reduced norm on $\mathbb{B}$. 
\end{lemma}

\begin{remark}
This lemma has been proved for non-optimal embeddings and for maximal orders $R$ in \cite[Proposition 12.9]{Gross_1987}. Probably for this reason, the lattice $R_0$ is sometimes called the \textit{Gross lattice} associated to $R$. The argument in \textit{loc. cit.} easily generalizes to our situation. We provide a full proof for completeness.
\end{remark}

\begin{proof}

We first prove that $\mathcal{O}$ embeds in $R$ if and only if $|\Delta|$ is represented by $R_0$, and we discuss conditions on the optimality of this embedding at a second stage.

Write $\mathcal{O}=\mathbb{Z}\left[ \frac{\Delta+\sqrt{\Delta}}{2} \right]$ and suppose first that $f: \mathcal{O} \hookrightarrow R$ is an embedding. Let $b:=f(\sqrt{\Delta})$ so that $\trd(b)=0$ and $\nrd(b)=|\Delta|$. Since
\[
f\left( \frac{\Delta+\sqrt{\Delta}}{2} \right)=\frac{\Delta+b}{2} \in R
\]
we see that $b\in R_0$ so that $|\Delta|$ is represented by this lattice. Suppose conversely that there exists $b \in R_0$ such that $\nrd(b)=|\Delta|$. Since $\trd(b)=0$, we see that $b^2=\Delta$. By writing $b=a+2r$ with $a\in \mathbb{Z}$ and $r\in R$, one has
\[
b^2=(a+2r)^2=a^2+4r^2+4ar=\Delta
\]
and this immediately implies 
that $a\equiv \Delta \text{ mod } 2$, so that $\Delta+b \in 2R$. Hence we have $(\Delta+b)/2 \in R$ and we obtain an embedding $f:\mathcal{O}\hookrightarrow R$ by setting
\begin{equation} \label{embedding}
f\left( \frac{\Delta+\sqrt{\Delta}}{2} \right)=\frac{\Delta+b}{2}.
\end{equation}

We now discuss optimality. Fix $\{\alpha_1,\alpha_2, \alpha_3 \}$ to be a basis of $R_0$ as a $\mathbb{Z}$-module and let $Q(X,Y,Z)$ be the ternary quadratic form induced by the reduced norm with respect to this basis.

Assume that $f: \mathcal{O} \hookrightarrow R$ is an optimal embedding. By the proof above, we know that  $b:=f(\sqrt{\Delta}) \in R_0$ is such that $\nrd(b)=|\Delta|$. Suppose by contradiction that $b=a_1 \alpha_1+a_2 \alpha_2+a_3 \alpha_3$ with $a_1,a_2,a_3 \in \mathbb{Z}$ not coprime, so that $c:= \gcd(a_1,a_2,a_3)>1$ (we adopt the convention that the greatest common divisor is always positive). Then $\widetilde{b}:=b/c \in R_0$ satisfies
\[
\widetilde{b}^2=\frac{\Delta}{c^2} \in \mathbb{Z} \hspace{0.5cm} \text{and} \hspace{0.5cm} \frac{1}{2} \left(\frac{\Delta}{c^2} + \widetilde{b}\right) \in R.
\]
in the same way as above. Thus $\frac{1}{2} \left( \frac{\Delta}{c^2}+\frac{\sqrt{\Delta}}{c}\right) \in K$ is an algebraic integer and the order $\widetilde{\mathcal{O}}:=\mathbb{Z} \left[ \frac{1}{2} \left(\frac{\Delta}{c^2}+\frac{\sqrt{\Delta}}{c}\right) \right]$, which strictly contains $\mathcal{O}$, also embeds in $R$ through the extension $f: K \hookrightarrow \mathbb{B}$. This contradicts the optimality of $f:\mathcal{O} \hookrightarrow R$.

Suppose now that $|\Delta|$ is primitively represented by $R_0$ \textit{i.e.} that there exist $a_1,a_2,a_3 \in \mathbb{Z}$ coprime such that $\nrd(a_1 \alpha_1+a_2 \alpha_2+a_3 \alpha_3)=|\Delta|$. We want to show that, setting $b:=a_1 \alpha_1+a_2 \alpha_2+a_3 \alpha_3$, the embedding $f$ defined by \eqref{embedding} is optimal. We will equivalently prove that, if $c\in \mathbb{Z}_{>0}$ is such that
$\widetilde{\mathcal{O}}:=\mathbb{Z} \left[ \frac{1}{2} \left(\frac{\sqrt{\Delta}}{c}+\frac{\Delta}{c^2}\right) \right]$ is an order, then
\begin{equation} \label{optimal embedding condition}
    f\left( K \right) \cap R = f\left( \widetilde{\mathcal{O}} \right)
\end{equation}

implies $\widetilde{\mathcal{O}}=\mathcal{O}$. Since $b=f(\sqrt{\Delta})$, equality \eqref{optimal embedding condition} entails $\frac{1}{2} \left(\frac{b}{c}+\frac{\Delta}{c^2}\right) \in R$ so that $b/c \in R_0$. But now
\[
b/c=\frac{a_1}{c} \alpha_1 + \frac{a_2}{c} \alpha_2 + \frac{a_3}{c} \alpha_3 \in R_0
\]
and all the coefficients $a_i/c$ must be integral since $\{\alpha_1, \alpha_2, \alpha_3 \}$ is a basis of $R_0$ as a $\mathbb{Z}$-module. By assumption, the $a_i$'s are coprime, so we must have $c=1$. Hence $\widetilde{\mathcal{O}}=\mathcal{O}$ and this concludes the proof.
\end{proof}

\begin{remark} \label{rmk:bijection_optimal_embeddings}
The proof of Lemma \ref{Gross Lattice} actually establishes a bijection between the set of embeddings $f:\mathcal{O} \hookrightarrow R$ and the set of elements $b \in R_0$ such that $\nrd(b)=|\Delta|$. Under this bijection, the embedding $f$ corresponds to the element $f(\sqrt{\Delta}) \in R_0$.
\end{remark}

In order to carry out our study of singular differences that are $S$-units, it is fundamental to understand what is the biggest exponent with which a prime ideal can appear in the factorization of such a difference. Roughly speaking, saying that a difference of singular moduli $j-j_0$ has a certain $\mu$-adic valuation $n=v_{\mu}(j-j_0)$ for some prime ideal $\mu \subseteq \mathbb{Q}(j-j_0)$ is equivalent to saying that the CM elliptic curve $E_j$ with $j(E_j)=j$ is isomorphic to the elliptic curve $E_{j_0}$ with $j(E_0)=j_0$ when reduced modulo $\mu^n$. Therefore, in order to understand the exponents appearing in the prime ideal factorization of a singular difference, it is crucial to determine when such isomorphisms can occur. With this goal in mind, we conclude this section by outlining some aspects of the reduction theory of CM elliptic curves defined over number fields. We refer the reader to \cite{Conrad_2004}, \cite{Gross_1986}, \cite{Gross_Zagier_1985} and \cite{Lauter_Viray_2015} for further discussions on the topic.


Let $\mathcal{O}$ be an order of discriminant $\Delta$ in an imaginary quadratic field $K$ and let $\ell \nmid \Delta$ be a prime inert in $K$. Consider an elliptic curve $E'$ with complex multiplication by the order $\mathcal{O}$ and defined over the ring class field $H_\mathcal{O}:=K(j(E'))$. After completing with respect to any prime above $\ell$, we can consider $H_\mathcal{O}$ as a subfield of the maximal unramified extension $\mathbb{Q}_{\ell}^{\text{unr}}$ of $\mathbb{Q}_\ell$. This is because the extension $\mathbb{Q} \subseteq H_\mathcal{O}$ is unramified at $\ell$ by the assumption $\ell \nmid \Delta$, see \cite[Chapter 9, Section A]{Cox_book_2013}. Let $L:=\widehat{\mathbb{Q}^{\text{unr}}_\ell}$ be the completion of $\mathbb{Q}_{\ell}^{\text{unr}}$ with ring of integers $W$ and uniformizer $\pi$. Then by \cite[Theorems 8 and 9]{Serre_Tate_1968} and \cite[Chapter 13, Theorem 12]{Lang_1987} there exists an elliptic scheme $\mathcal{E} \to \Spec W$ such that:
\begin{itemize}
    \item the generic fiber $E:=\mathcal{E} \times_W \Spec L$ is isomorphic to $E'$ over the algebraic closure of $L$. Since the CM order $\mathcal{O}$ is contained in $W$, all the geometric endomorphisms of $E$ are defined over $L$, see \cite[Chapter II, Proposition 30]{Shimura_book_1998};
    \item the special fiber $E_0:=\mathcal{E} \times_W \Spec W/\pi$ is a supersingular elliptic curve since, by assumption, $\ell$ does not split in $K$. Note that $W/\pi \cong \overline{\mathbb{F}}_\ell$, the algebraic closure of the finite field with $\ell$ elements.
\end{itemize}
For all $n\in \mathbb{N}$, set $E_n:=\mathcal{E} \times_W \Spec W/ \pi^{n+1}$. We are interested in understanding the endomorphism rings $A_{\ell,n}:=\End_{W/\pi^{n+1}}(E_n)$. When $n=0$, we have already seen that the ring $A_{\ell,0}$ is isomorphic to a maximal order in $\mathbb{B}_{\ell, \infty}$, the unique (up to isomorphism) definite quaternion algebra over the rationals which ramifies only at $\ell$ and $\infty$. All the other rings $A_{\ell,n}$ can be recovered from $A_{\ell,0}$, as explained in the following theorem.

\begin{theorem} \label{deformation}
Let $\mathcal{O}$ be an order of discriminant $\Delta$ in an imaginary quadratic field $K$ and let $\ell \nmid \Delta$ be a prime inert in $K$. Set $L:=\widehat{\mathbb{Q}^{\text{unr}}_\ell}$ to be the completion of the maximal unramified extension of $\mathbb{Q}_\ell$, with ring of integers $W$ and uniformizer $\pi$. Let $\mathcal{E} \to \Spec (W)$ be an elliptic scheme whose generic fiber $E:= \mathcal{E} \times_W \Spec L$ has complex multiplication by $\mathcal{O}$. For every $n \in \mathbb{N}$, denote by 
\[
E_n:=\mathcal{E} \times_W \Spec W/\pi^{n+1} \hspace{0.5cm} \text{and} \hspace{0.5cm} A_{\ell,n}:=\End_{W/\pi^{n+1}}(E_n)
\]
respectively the reduction of $\mathcal{E}$ modulo $\pi^{n+1}$ and its endomorphism ring. Then:
\begin{enumerate}
\item[(a)] for every $n \in \mathbb{N}$ we have
        \[
          A_{\ell,n}  \cong  \mathcal{O} + \ell^n A_{\ell,0},
         \]
    where the sum takes place in $A_{\ell,0}$ in which $\mathcal{O}$ is embedded via the reduction modulo $\pi$;
    \item[(b)] for every $n \in \mathbb{N}$ the ring $\End_{W/\pi^{n+1}}(E_n)$ is isomorphic to a quaternion order in $\mathbb{B}_{\ell, \infty}$ and the natural reduction map
    \[
    \mathcal{O} \cong \End_W (\mathcal{E}) \longrightarrow \End_{W/\pi^{n+1}}(E_n)
    \]
    induced by the reduction modulo $\pi^{n+1}$ is an optimal embedding.
\end{enumerate}

\end{theorem}
The above theorem is a combination and a reformulation of various results already appearing in the literature. We give a brief overview of the proof and point out the relevant references.

\begin{proof}[Proof of Theorem \ref{deformation}]
 Part (a) of the theorem is a special case of \cite[Formula 6.6]{Lauter_Viray_2015}. As for part (b): the first statement follows from the fact that $\ell$ is a prime of supersingular reduction for $E$ and from part (a). For the second statement, note first of all that there is a natural isomorphism between $\End_L(E) \cong \mathcal{O}$ and $\End_W(\mathcal{E})$, since by assumption $\mathcal{E}$ is a Néron model for $E$ over $W$ (see \cite[Propositions 1.2/8 and 1.4/4]{Bosch_Neron_models_book_1990}). Reductions modulo $\pi$ and $\pi^n$ give the following commutative diagram

\begin{center}
    \begin{tikzcd}
\mathcal{O} \arrow[r, "\varphi_{n-1}", hook] \arrow[rd, "\varphi_0"', hook] & \End_{W/\pi^n}(E_{n-1}) \arrow[d, hook] \\
                                                                  & \End_{W/\pi}(E_0)                
\end{tikzcd}
\end{center}

in which all the arrows are injective by \cite[Theorem 2.1 (2)]{Conrad_2004}. Since $\ell$ does not divide the conductor of the order $\mathcal{O}$, the embedding $\varphi_0$ is optimal by \cite[Proposition 2.2]{Lauter_Viray_2015}. It follows from the commutativity of the diagram above that also the embedding $\varphi_{n-1}$ is optimal, and the theorem is proved. 
\end{proof}

\section{The $\ell$-adic valuation of differences of singular moduli} \label{Section 4}

In order to bound from above the Weil height of a difference of singular moduli, it is of crucial importance to understand the exponents appearing in the prime factorization of such a difference. The goal of this section is to prove, under certain conditions, an upper bound for these exponents. In what follows, we will always use $\mathbb{F}_\ell$ to denote the finite field with $\ell$ elements, where $\ell \in \mathbb{N}$ is a prime number, and denote by $\overline{\mathbb{F}}_\ell$ an algebraic closure of this field. Recall also that given an order $\mathcal{O}$ in an imaginary quadratic field $K$, the \textit{ring class field} of $K$ relative to the order $\mathcal{O}$ is the field generated over $K$ by any singular modulus relative to $\mathcal{O}$.

\begin{theorem} \label{Theorem valuation j-j0}
Let $j_0 \in \overline{\mathbb{Q}}$ be a singular modulus relative to an order $\mathcal{O}_{j_0}$ of discriminant $\Delta_0$ and let $\ell \in \mathbb{Z}$ be a prime not dividing $\Delta_0$. For any singular modulus $j \in \overline{\mathbb{Q}}$ relative to an order $\mathcal{O}_j$ of discriminant $\Delta \neq \Delta_0$, denote by $H$ the compositum of the ring class fields relative to $\mathcal{O}_{j_0}$ and $\mathcal{O}_j$. Let $\mu \subseteq H$ be a prime ideal lying above $\ell$ and assume that 
\begin{enumerate}
    \item the prime $\mu \cap \mathbb{Q}(j_0)$ has residue degree $1$ over $\ell$;
    \item there exists an elliptic curve ${E_0}_{/\mathbb{Q}(j_0)}$ with $j(E_0)=j_0$ and having good reduction at $\mu \cap \mathbb{Q}(j_0)$.
\end{enumerate}
Then, if $v_\mu(\cdot)$ denotes the normalized valuation associated to $\mu$, we have
\begin{equation} \label{valuation of j-j0}
    v_\mu(j-j_0) \leq 
\begin{cases}
\frac{d_0}{2} \left( \frac{\log(\Delta_0^2|\Delta|)}{2 \log \ell}+ \frac{1}{2} \right) &\text{ if } \ell \nmid \Delta \text{ and } \mathcal{O}_{j_0} \not \subseteq \mathcal{O}_j, \\
\frac{d_0}{2} &\text{ if } \ell \mid \Delta
\end{cases}
\end{equation}
where $d_0$ is the number of automorphisms of any elliptic curve $E_{/\overline{\mathbb{F}}_\ell}$ with $j(E) = j_0 \text{ mod } \mu$.
\end{theorem}

\begin{remark}
Note that we have $d_0=2$ in all cases except if $j_0 \equiv 0$ or $j_0 \equiv 1728$ mod $\mu$. In these two cases, the value of $d_0$ also depends on $\ell$, see \cite[III, Theorem 10.1]{Silverman_2009}.
\end{remark}

The dichotomy in the conclusion of Theorem \ref{Theorem valuation j-j0} is reflected by its proof, which we divide according to the conditions displayed in \eqref{valuation of j-j0}. In all cases, everything boils down to the study of optimal embeddings of the order $\mathcal{O}_j$ in a family of nested orders contained in the endomorphism ring of a certain supersingular elliptic curve defined over $\overline{\mathbb{F}}_\ell$. One of the main issues is that for a supersingular elliptic curve $E_{/\overline{\mathbb{F}}_\ell}$, explicitely computing its endomorphism ring is a difficult problem in general. An explicit parametrization of the endomorphism rings of supersingular elliptic curves over $\overline{\mathbb{F}}_\ell$ has been achieved by Lauter and Viray in \cite[Section 6]{Lauter_Viray_2015}. However, the author found these parametrizations somehow difficult to use for explicit estimates. Therefore, in order to achieve our results, we adopted a different strategy. The idea is that, since we are only interested in providing estimates for the $\mu$-adic valuation of singular differences and not in precisely determining their prime ideal factorization, we do not need the full knowledge of the supersingular endomorphism rings of the elliptic curves involved. We instead "approximate", when possible, the unknown quaternion orders with quaternion orders whose properties are less mysterious. The next proposition is the cornerstone of this strategy.

\begin{proposition} \label{good element in the quaternion algebra}
Let $j \in \overline{\mathbb{Q}}$ be a singular modulus of discriminant $\Delta$ and let $E_{/\mathbb{Q}(j)}$ be an elliptic curve with $j(E)=j$. Choose a degree $1$ prime $\mathfrak{p} \subseteq \mathbb{Q}(j)$ lying above a rational prime $p \in \mathbb{Z}$ not dividing $\Delta$ and suppose that $E$ has good supersingular reduction $\widetilde{E}$ modulo $\mathfrak{p}$. Denote by $\varphi \in \End_{\overline{\mathbb{F}}_p}(\widetilde{E})$ the Frobenius endomorphism $(x,y) \mapsto (x^p,y^p)$, where the coordinates $x,y$ come from the choice of a Weierstrass model for $E$. Then there exists a morphism $\psi \in \End_{\overline{\mathbb{F}}_p}(\widetilde{E})$ such that
\[
\psi^2+|\Delta| \psi+ \frac{\Delta^2 + |\Delta|}{4}=0 \hspace{0.5cm} \text{and} \hspace{0.5cm} \psi \circ \varphi = \varphi \circ \overline{\psi}
\]
where $\overline{\cdot}:\End_{\overline{\mathbb{F}}_p}(\widetilde{E}) \otimes_\mathbb{Z} \mathbb{Q} \to \End_{\overline{\mathbb{F}}_p}(\widetilde{E}) \otimes_\mathbb{Z} \mathbb{Q}$ denotes the standard involution. In fact, the morphism $\psi$ can be taken inside the image of the reduction map $\End_{\overline{\mathbb{Q}}}(E) \to \End_{\overline{\mathbb{F}}_p}(\widetilde{E})$ modulo any prime in $\overline{\mathbb{Q}}$ lying above $\mathfrak{p}$.
\end{proposition}

\begin{remark}
Recall that the standard involution on the quaternion algebra $\End_{\overline{\mathbb{F}}_p}(\widetilde{E}) \otimes_\mathbb{Z} \mathbb{Q}$ corresponds to taking the dual isogeny when restricted to $\End_{\overline{\mathbb{F}}_p}(\widetilde{E})$. This essentially follows from the uniqueness of the standard involution on quaternion algebras, see \cite[Corollary 3.4.4]{Voight_book_2020}.
\end{remark}

\begin{proof}
 In this proof, we fix for convenience an embedding $\overline{\mathbb{Q}} \hookrightarrow \mathbb{C}$. Let $\mathcal{O}$ be the order of discriminant $\Delta$ and $K \subseteq \overline{\mathbb{Q}}$ be its field of fractions.
For an element $\beta \in K$, we denote by $\overline{\beta}$ its conjugate through the unique non-trivial automorphism of $K/\mathbb{Q}$. This will not cause confusion with the standard involution on $\End_{\overline{\mathbb{F}}_p}(\widetilde{E}) \otimes_\mathbb{Z} \mathbb{Q}$, as we explain below. 
 
 By assumption, there exists an elliptic scheme $\mathcal{E}$ over the localization at $\mathfrak{p}$ of the ring of integers in $\mathbb{Q}(j)$ such that the generic fiber of $\mathcal{E}$ is isomorphic to $E$ while its special fiber is a supersingular elliptic curve $\widetilde{E}$ defined over $\mathbb{F}_p$. Set $H_\mathcal{O}:=K(j)$, which is a degree $2$ extension of $\mathbb{Q}(j)$, and fix a prime $\mathcal{P} \subseteq H_\mathcal{O}$ lying above $\mathfrak{p}$.
 Since $E$ has supersingular reduction modulo $\mathfrak{p}$, the latter has degree $1$ and $p$ is unramified in $K$, by \cite[Chapter 13, Theorem 12]{Lang_1987} we must have $f(\mathcal{P}/\mathfrak{p})=2$, where $f(\mathcal{P}/\mathfrak{p})$ denotes the inertia degree of $\mathcal{P}$ over $\mathfrak{p}$. In particular, we see that the decomposition group of $\mathcal{P}$ over $\mathfrak{p}$ is precisely $\Gal(H_\mathcal{O}/\mathbb{Q}(j))$. We fix $\sigma \in \Gal(H_\mathcal{O}/\mathbb{Q}(j))$ to be the unique non-trivial element. Then $\sigma$ restricts to an automorphism of $R_\mathcal{P}$, the localization at $\mathcal{P}$ of the ring of integers of $H_\mathcal{O}$, inducing the Frobenius endomorphism $\tau: x \mapsto x^p$ on the residue field.
 
  With abuse of notation, we denote again by $\mathcal{E}$ the base-change $\mathcal{E}_{R_\mathcal{P}}$ and by $\widetilde{E}$ the special fiber of $\mathcal{E}_{R_\mathcal{P}}$ (which is isomorphic to the base-change of the special fiber of $\mathcal{E}$ to the residue field of $R_\mathcal{P}$).
  It follows from the Néron mapping property \cite[Proposition 1.4/4]{Bosch_Neron_models_book_1990} that every endomorphism $\lambda \in \End_{H_\mathcal{O}}(E)$ induces an endomorphism $\lambda_\mathcal{E}$ of $\mathcal{E}$. Define $\lambda \text{ mod } \mathcal{P}$ to be the restriction of $\lambda_\mathcal{E}$ to $\widetilde{E}$. The Galois group $\Gal(H_\mathcal{O}/\mathbb{Q}(j))$ acts on $R_\mathcal{P}$ and this in turn induces a Galois action on $\End_{R_\mathcal{P}}(\mathcal{E})$. In the same way, there is an action of $\Gal(\overline{\mathbb{F}}_p/\mathbb{F}_p)$ on $\End_{\overline{\mathbb{F}}_p}(\widetilde{E})$. These two actions are compatible, in the sense that for every $\lambda \in \End_{H_\mathcal{O}}(E)$ we have
 \begin{equation} \label{eq:compatibility_Galois_actions}
     \sigma(\lambda_\mathcal{E}) \text{ mod } \mathcal{P} = \tau (\lambda \text{ mod } \mathcal{P}),
 \end{equation}
 as one can see using the various functorial properties of fibered products.
 In what follows, we will often omit the subscript $\mathcal{E}$ when dealing with endomorphism of $\mathcal{E}$ induced by elements in $\End_{H_\mathcal{O}}(E)$. This allows us to ease a  bit the notation, since usually elements of $\End_{H_\mathcal{O}}(E)$ will already come equipped with their own subscript.
 
We now fix a normalized isomorphism
 \[
 [\cdot]_E: \mathcal{O} \xrightarrow{\sim} \End_{\overline{\mathbb{Q}}}(E) 
 \]
 following \cite[II, Proposition 1.1]{Silverman_book_1994}. 
 Let $\alpha:= \frac{\Delta+\sqrt{\Delta}}{2} \in \mathcal{O}$ and note that $[\alpha]_E \in \End_{H_\mathcal{O}}(E)$ because, by \cite[Chapter II, Proposition 30]{Shimura_book_1998}, all the endomorphisms of $E$ are defined over $H_\mathcal{O}$. Since $\alpha^2+|\Delta| \alpha+ \frac{\Delta^2 + |\Delta|}{4}=0$, also $[\alpha]_E$ satisfies the same relation. One also has
 \begin{equation} \label{eq:sigma_of_morphism}
     \sigma([\alpha]_E)=[\sigma(\alpha)]_{E^{\sigma}}=[\overline{\alpha}]_E
 \end{equation}
 where the first equality follows from \cite[II, Theorem 2.2 (a)]{Silverman_book_1994} and in the second equality we are using the fact that $E$ is defined over $\mathbb{Q}(j)$ and $\sigma$ is non-trivial.
 
 Let now $\psi:=\left( [\alpha]_E \text{ mod } \mathcal{P} \right) \in \End_{\overline{\mathbb{F}}_p}(\widetilde{E})$. For $\beta \in \mathcal{O}$, the association $[\beta]_E \mapsto [\overline{\beta}]_E$ defines a standard involution on $\End_{H_\mathcal{O}}(E)$, in the sense of \cite[Definition 3.2.4]{Voight_book_2020}. Since reduction mod $\mathcal{P}$ defines an embedding of $\End_{H_\mathcal{O}}(E) \hookrightarrow \End_{\overline{\mathbb{F}}_p}(\widetilde{E})$, by the uniqueness of the standard involution on quadratic $\mathbb{Q}$-algebras (see \cite[Lemma 3.4.2]{Voight_book_2020}) we have $[\overline{\alpha}]_E \text{ mod } \mathcal{P}= \overline{\psi}$, where now the conjugation above $\psi$ denotes the usual standard involution on the quaternion algebra $\End_{\overline{\mathbb{F}}_p}(\widetilde{E}) \otimes_\mathbb{Z} \mathbb{Q}$. We have
 \[
 \overline{\psi} = [\overline{\alpha}]_E \text{ mod } \mathcal{P} = \sigma([\alpha]_E) \text{ mod } \mathcal{P} = \tau ( [\alpha]_E \text{ mod } \mathcal{P} ) = \tau (\psi)
 \]
 where we have applied equalities \eqref{eq:compatibility_Galois_actions} and \eqref{eq:sigma_of_morphism}. This yields
 \[
 \varphi \circ \overline{\psi} = \varphi \circ \tau (\psi) = \tau (\tau (\psi)) \circ \varphi = \psi \circ \varphi
 \]
 and here we have used the facts that for every $\lambda \in \End_{\overline{\mathbb{F}}_p}(\widetilde{E})$ one has $\varphi \circ \lambda = \tau(\lambda) \circ \varphi$, as can be checked using local coordinates for $\widetilde{E}$, and that $\tau(\tau(\psi))=\psi$ because $\psi$ is defined over a quadratic extension of $\mathbb{F}_p$. The proof is concluded.
\end{proof}

We are now ready to begin the proof of Theorem \ref{Theorem valuation j-j0}.
Let us fix the notation that will be in force during the entire argument. Given the orders $\mathcal{O}_j=\mathbb{Z} \left[ \frac{\Delta+\sqrt{\Delta}}{2} \right]$ and $\mathcal{O}_{j_0}=\mathbb{Z} \left[ \frac{\Delta_0+\sqrt{\Delta_0}}{2} \right]$ as in the statement of Theorem \ref{Theorem valuation j-j0}, we denote by $K_j$ and $K_{j_0}$ the corresponding imaginary quadratic fields containing them. We then set $H_j$ and $H_{j_0}$ to be the ring class fields of $K_j$ and $K_{j_0}$ relative to the orders $\mathcal{O}_j$ and $\mathcal{O}_{j_0}$ respectively. Using this notation, the field $H$ in the statement of Theorem \ref{Theorem valuation j-j0} is the compositum in $\overline{\mathbb{Q}}$ of $H_j$ and $H_{j_0}$.

\subsection{First case: $\ell$ does not divide $\Delta$ and $\mathcal{O}_{j_0} \not \subseteq \mathcal{O}_j$.}

Assume that $E_0$ in the statement of the theorem is given by an integral model over the ring of integers of $\mathbb{Q}(j_0)$ with good reduction at $\mu \cap \mathbb{Q}(j_0)$. Let $(E_0)_{/H}$ be the base-change to $H$ of the elliptic curve $(E_0)_{/\mathbb{Q}(j_0)}$, and let $(E_j)_{/H}$ be an elliptic curve with $j(E_j)=j$ and with good reduction at all prime ideals above $\ell$. Such an elliptic curve $E_j$ exists by \cite[Theorems 8 and 9]{Serre_Tate_1968}, which we can apply since $\ell \nmid \Delta$ by assumption. In particular, $E_j$ will have good reduction at the prime $\mu$. We will always identify $\mathcal{O}_j$ and $\mathcal{O}_{j_0}$ with the endomorphism rings of $E_j$ and $E_0$ respectively.

Let $H_\mu$ be the completion of $H$ at the prime $\mu$. The extension $\mathbb{Q} \subseteq H$ is unramified at $\ell$ because $\ell \nmid \Delta \Delta_0$ (see \cite[Chapter 9, Section A]{Cox_book_2013}), hence $H_\mu$ is contained in $\widehat{\mathbb{Q}^{\text{unr}}_\ell}$, the completion of the maximal unramified extension of $\mathbb{Q}_\ell$. Denote by  $W$ the ring of integers in $\widehat{\mathbb{Q}^{\text{unr}}_\ell}$ and let $\pi \in W$ be a uniformizer. By abuse of notation, we also use $E_0, E_j$ to denote the elliptic schemes over $W$ with generic fibers isomorphic to the base-changes of $E_0, E_j$ to $\widehat{\mathbb{Q}^{\text{unr}}_\ell}$ respectively. Note that, by our choices, $E_0 \text{ mod } \pi$ is defined over $\mathbb{F}_\ell$.

\begin{lemma} \label{Precise valuation}
In the notation above, we have 
\[
v_\mu(j-j_0) \leq \frac{d_0}{2} \cdot \max \{n \in \mathbb{N}_{\geq 1}: \Iso_{W/\pi^n} (E_j, E_0) \neq \emptyset \}
\]
where, for every $n \in \mathbb{Z}_{\geq 1}$, we denote by $\Iso_{W/\pi^n} (E_j, E_0)$ the set of isomorphisms between $E_j \text{ mod } \pi^n$ and $E_0 \mod \pi^n$. 
\end{lemma}

\begin{proof}
Notice first of all that the normalized valuation on $\widehat{\mathbb{Q}^{\text{unr}}_\ell}$, \textit{i.e} the valuation $v$ satisfying $v(\pi)=1$, extends the $\mu$-adic valuation $v_\mu$ on $H$ because $v_\mu(\ell)=1$. Since $W$ is a complete discrete valuation ring whose quotient field has characteristic $0$ and whose residue field $\overline{\mathbb{F}}_\ell$ is algebraically closed of characteristic $\ell >0$, we can apply \cite[Proposition 2.3]{Gross_Zagier_1985} which gives
\[
v_\mu(j-j_0)= \frac{1}{2} \sum_{n=1}^{\infty} \# \Iso_{W/\pi^n} (E_j, E_0).
\]
Now, certainly $\Iso_{W/\pi^{n+1}} (E_j, E_0) \neq \emptyset$ implies $\Iso_{W/\pi^{n}} (E_j, E_0) \neq \emptyset$ for every $n \in \mathbb{N}_{>0}$, since reductions of isomorphisms are isomorphisms. Moreover, whenever the set $\Iso_{W/\pi^n} (E_j, E_0)$ is non-empty, its cardinality equals the order of the automorphism group $\Aut_{W/\pi^n}(E_0)$ of $E_0 \text{ mod } \pi^n$. 
By \cite[Theorem 2.1 (2)]{Conrad_2004}, we always have the inclusions
\[
\End_W(E_0) \hookrightarrow \End_{W/\pi^n}(E_0) \hookrightarrow \End_{W/\pi}(E_0)
\]
induced respectively by the reduction modulo $\pi^n$ and modulo $\pi$. This means that 
\begin{equation} \label{inequalities automorphisms}
    \# \Aut_W(E_0) \leq \# \Aut_{W/\pi^n}(E_0) \leq \# \Aut_{W/\pi}(E_0)=d_0
\end{equation}
so, setting $M:=\max \{n \in \mathbb{N}_{\geq 1}: \Iso_{W/\pi^n} (E_j, E_0) \neq \emptyset \}$, we obtain
\[
v_\mu(j-j_0)= \frac{1}{2} \sum_{n=1}^{\infty} \# \Iso_{W/\pi^n} (E_j, E_0)=\frac{1}{2} \sum_{n=1}^{M} \# \Aut_{W/\pi^n}(E_0) \leq \frac{d_0}{2}\cdot M
\]
which proves the lemma. 

\end{proof}

By Lemma \ref{Precise valuation}, in order to estimate the valuation at $\mu$ of the difference $j-j_0$, we need to bound the biggest index $n$ such that the reductions modulo $\pi^n$ of the elliptic curves $E_j$ and $E_0$ are isomorphic. If this maximum is $0$, then the two elliptic curves are not even isomorphic over $\overline{\mathbb{F}}_\ell \cong W/\pi$, so the prime $\mu$ cannot divide $j-j_0$ and there is nothing to prove. Hence, from now on we suppose that $\mu$ divides $j-j_0$ so that $E_0 \text{ mod } \pi \cong E_j \text{ mod } \pi$ over $\overline{\mathbb{F}}_\ell$. Since $\ell$ does not divide the conductors of the orders $\mathcal{O}_j$ and $\mathcal{O}_{j_0}$ by assumption, and the two orders are different, \cite[Chapter 13, Theorem 12]{Lang_1987} ensures that $\ell$ is a prime of supersingular reduction for both $E_j$ and $E_{0}$. In particular, the ring $R:=\End_{W/\pi}(E_0)$ is isomorphic to a maximal order in $\mathbb{B}_{\ell, \infty}\cong R \otimes_\mathbb{Z} \mathbb{Q}$.

Suppose now that $\Iso_{W/\pi^{n+1}}(E_j,E_0)$ is non-empty. Our goal is to find a bound on the exponent $n+1$. A choice of $f \in \Iso_{W/\pi^{n+1}}(E_j,E_0)$ induces an isomorphism
\[
\widetilde{f}: \End_{W/\pi^{n+1}}(E_j) \to \End_{W/\pi^{n+1}}(E_0), \hspace{0.5cm} \alpha \mapsto f \circ \alpha \circ f^{-1}
\] 
which, precomposed with the reduction map $\mathcal{O}_j \hookrightarrow \End_{W/\pi^{n+1}}(E_j)$, gives rise to an optimal embedding
\begin{equation} \label{injection}
    \psi_{n+1}: \mathcal{O}_j \hookrightarrow \End_{W/\pi^{n+1}}(E_0)
\end{equation}
by Theorem \ref{deformation} (b). For growing $n$, Theorem \ref{deformation} (a) shows that the endomorphism ring of $E_0 \text{ mod } \pi^{n+1}$ becomes more and more "$\ell$-adically close" to the order $\mathcal{O}_{j_0}$. Intuitively, this must imply that having an embedding as in \eqref{injection} should not be possible for $n$ large enough, yielding the desired bound on $n+1$. This intuition is correct, as we show below. The main obstacle to making this idea precise is that, as we already said, it is not easy to explicitly compute the endomorphism rings $\End_{W/\pi^{n+1}}(E_0)$ for a generic elliptic curve ${E_0}_{/W}$. To circumvent this problem, we "approximate" the rings $\End_{W/\pi^{n+1}}(E_0)$ with smaller orders where we are able to perform the relevant computations. The hypotheses on the prime $\mu$ and on the elliptic curve $E_0$ will make this strategy successful. 

Recall that $\mathcal{O}_{j_0}=\mathbb{Z} \left[ \frac{\Delta_0 + \sqrt{\Delta_0}}{2} \right]$ and let $\psi \in R$ be the image of $\frac{\Delta_0 + \sqrt{\Delta_0}}{2}$ via the reduction map modulo $\pi$. Denote also by $\varphi \in \End_{W/\pi}(E_0)$ the Frobenius endomorphism $(x,y) \mapsto (x^\ell, y^\ell)$. By Proposition \ref{good element in the quaternion algebra} and using the fact that $E_0 \text{ mod } \pi$ is a supersingular elliptic curve defined over $\mathbb{F}_\ell$, we have
\begin{equation} \label{relations in the order}
    \varphi^2+\ell=0, \hspace{0.5cm} \psi^2+|\Delta_0| \psi+ \frac{\Delta_0^2 + |\Delta_0|}{4}=0 \hspace{0.2cm} \text{and} \hspace{0.2cm} \psi \circ \varphi = \varphi \circ \overline{\psi}
\end{equation}
where $\overline{\cdot}$ denotes the standard involution on $\End_{W/\pi}(E_0) \otimes_\mathbb{Z} \mathbb{Q}$.
Hence, the ring $\widetilde{R}:=\mathbb{Z}[ \psi, \varphi] \subseteq R$ is a rank-$4$ order inside $\mathbb{B}_{\ell, \infty}$ with basis $\mathcal{B}=\{1, \psi, \varphi, \psi \varphi \}$ satisfying the relations \eqref{relations in the order}. Notice that the reduction map $\mathcal{O}_{j_0} \hookrightarrow R$ identifies $\mathcal{O}_{j_0}$ with the subring $\mathbb{Z}[\psi] \subseteq \mathbb{Z}[\psi, \varphi]$. The matrix of the bilinear pairing $\langle \alpha, \beta \rangle=\trd(\alpha \overline{\beta})$ computed on the basis $\mathcal{B}$ is given by
\[
A=\begin{pmatrix}
2 & \Delta_0 & 0 & 0 \\
\Delta_0 & \frac{\Delta_0^2 + |\Delta_0|}{2} & 0 & 0 \\
0 & 0 & 2\ell & \Delta_0 \ell \\
0 & 0 & \Delta_0 \ell & \frac{\Delta_0^2 + |\Delta_0|}{2} \ell \end{pmatrix}
\]

so the discriminant of the order $\widetilde{R}$ equals $\det A=\Delta_0^2 \ell^2$ (see \cite[Definition 15.2.2 and Exercise 13 in Chapter 15]{Voight_book_2020}). Hence, by \cite[Lemma 15.2.15, Lemma 15.4.7 and Theorem 15.5.5]{Voight_book_2020} $\widetilde{R}$ has index $|\Delta_0|$ inside any maximal order containing it, so in particular $|R:\widetilde{R}|=|\Delta_0|$. Now, since we are in the hypotheses of Theorem \ref{deformation} (a), we have
\[
\End_{W/\pi^{n+1}}(E_0) \cong \mathbb{Z}[\psi] + \ell^n R \supseteq \mathbb{Z}[\psi] + \ell^n \widetilde{R}
\]
and we shall show that the index of the latter inclusion is also bounded by $|\Delta_0|$.

\begin{lemma} \label{lemma index}
For all $n\in \mathbb{N}$ the index $\left| (\mathbb{Z}[\psi] + \ell^n R) : (\mathbb{Z}[\psi] + \ell^n \widetilde{R}) \right|$ divides $|\Delta_0|$.
\end{lemma}

\begin{proof}
 Since $\widetilde{R} \subseteq R$, we have $\mathbb{Z}[\psi] + \ell^n R = \mathbb{Z}[\psi] + \ell^n R + \ell^n \widetilde{R}$. Hence 
 \[
 \frac{\mathbb{Z}[\psi] + \ell^n R}{\mathbb{Z}[\psi] + \ell^n \widetilde{R}} = \frac{\mathbb{Z}[\psi] + \ell^n R + \ell^n \widetilde{R}}{\mathbb{Z}[\psi] + \ell^n \widetilde{R}} \cong \frac{\ell^n R}{(\mathbb{Z}[\psi] + \ell^n \widetilde{R}) \cap \ell^n R}
 \]
 as abelian groups. Now, the containment $\ell^n \widetilde{R} \subseteq (\mathbb{Z}[\psi] + \ell^n \widetilde{R}) \cap \ell^n R$ gives an epimorphism
 \[
 \frac{\ell^n R}{\ell^n \widetilde{R}} \twoheadrightarrow \frac{\ell^n R}{(\mathbb{Z}[\psi] + \ell^n \widetilde{R}) \cap \ell^n R}.
 \]
 and, since $R$ is non-torsion, we have $\ell^n R / \ell^n \widetilde{R} \cong R/\widetilde{R}$. Since the latter has cardinality $|\Delta_0|$, the lemma is proved.
 
\end{proof}

\begin{corollary} \label{embedding in a smaller order}
The embedding \eqref{injection} induces an injection
\begin{equation} \label{equation embedding in a smaller order}
    \mathcal{O}_{j,|\Delta_0|}:=\mathbb{Z} \left[  \frac{\Delta_0^2 \Delta+\sqrt{\Delta_0^2 \Delta}}{2} \right] \hookrightarrow \mathbb{Z}[\psi] + \ell^n \widetilde{R}.
\end{equation}
\end{corollary}

\begin{proof}
 By Lemma \ref{lemma index}, for every $x \in \mathbb{Z}[\psi] + \ell^n R$ we have $|\Delta_0| x \in \mathbb{Z}[\psi] + \ell^n \widetilde{R}$. Since $\mathcal{O}_{j,|\Delta_0|}=\mathbb{Z}+|\Delta_0| \mathcal{O}_j$, the corollary follows.
\end{proof}

Combining Corollary \ref{embedding in a smaller order} with Lemma \ref{Gross Lattice}, we see that $|\disc(\mathcal{O}_{j,|\Delta_0|})|=\Delta_0^2 |\Delta|$ must be represented by the Gross lattice $\Lambda_{\ell,n}$ of the order $\mathbb{Z}[\psi] + \ell^n \widetilde{R}$. Note that this representation is not necessarily primitive, because the embedding \eqref{equation embedding in a smaller order} is not necessarily optimal. A computation shows that 
\[
\Lambda_{\ell,n} = \left \langle |\Delta_0| + 2\psi, 2\ell^n \varphi, 2 \ell^n \psi \varphi \right \rangle_{\mathbb{Z}}
\]
\textit{i.e.} $\mathcal{B}'=\{|\Delta_0| + 2\psi, 2\ell^n \varphi, 2 \ell^n \psi \varphi \}$ is a $\mathbb{Z}$-basis for the Gross lattice of $\mathbb{Z}[\psi] + \ell^n \widetilde{R}$. The reduced norm restricted to the lattice $\Lambda_{\ell,n}$ induces the ternary quadratic form
\begin{equation} \label{Quadratic form}
    Q_{\ell,n}(X,Y,Z)=|\Delta_0|X^2+ 4\ell^{2n+1} Y^2 + \ell^{2n+1}(\Delta_0^2 + |\Delta_0|) Z^2+ 4\ell^{2n+1}\Delta_0 YZ
\end{equation}
written with respect to the basis $\mathcal{B}'$. 

After setting 
\[
\widetilde{X}=X, \hspace{0.5cm} \widetilde{Y}=Y+\frac{1}{2}\Delta_0 Z, \hspace{0.5cm} \widetilde{Z}=Z 
\]
we get the diagonal quadratic form 
\[
\widetilde{Q}_{\ell,n} (\widetilde{X}, \widetilde{Y}, \widetilde{Z})=|\Delta_0|\widetilde{X}^2+4\ell^{2n+1} \widetilde{Y}^2+\ell^{2n+1} |\Delta_0| \widetilde{Z}^2.
\]
Suppose now that $Q_{\ell,n}(X,Y,Z)=\Delta_0^2|\Delta|$ has an integral solution $(x,y,z) \in \mathbb{Z}^3$ corresponding to the embedding  \eqref{equation embedding in a smaller order}. We first claim that at least one among $y$ and $z$ is non-zero. This follows from our assumptions on $\mathcal{O}_j$ and from the following proposition.

\begin{proposition}
If $y=z=0$ then $\mathcal{O}_{j_0} \subseteq \mathcal{O}_j$.
\end{proposition}

\begin{proof}
 Let $x \in \mathbb{Z}_{>0}$ be such that $Q_{\ell,n}(x,0,0)=\Delta_0^2|\Delta|$.
 By Remark \ref{rmk:bijection_optimal_embeddings}, this equality corresponds to the embedding
 \begin{equation} \label{embedding 1}
     \mathbb{Z} \left[ \frac{1}{2}\left(\Delta_0^2 \Delta + \sqrt{\Delta_0^2 \Delta} \right) \right] \hookrightarrow \mathbb{Z}[\psi] + \ell^n \widetilde{R}, \hspace{0.5cm} \frac{1}{2}\left(\Delta_0^2 \Delta + \sqrt{\Delta_0^2 \Delta} \right) \mapsto \frac{1}{2}\left(\Delta_0^2 \Delta + x(|\Delta_0|+2\psi)\right)
 \end{equation}
 of the order $\mathcal{O}_{j,|\Delta_0|}\subseteq K:=\mathbb{Q}(\sqrt{\Delta})$ into $\mathbb{Z}[\psi] + \ell^n \widetilde{R}$. The injection \eqref{embedding 1} is not optimal if $x \neq \pm 1$. Indeed, using the proof of Lemma \ref{Gross Lattice} we get the optimal embedding
 \begin{equation*} 
     \mathbb{Z} \left[ \frac{1}{2}\left(\frac{\Delta_0^2}{x^2} \Delta + \sqrt{\frac{\Delta_0^2}{x^2} \Delta} \right) \right] \hookrightarrow \mathbb{Z}[\psi] + \ell^n \widetilde{R}, \hspace{0.5cm} \frac{1}{2}\left(\frac{\Delta_0^2}{x^2} \Delta + \sqrt{\frac{\Delta_0^2}{x^2} \Delta} \right) \mapsto \frac{1}{2}\left(\frac{\Delta_0^2}{x^2} \Delta + (|\Delta_0|+2\psi)\right)
 \end{equation*}
 determined by the equality $Q_{\ell,n}(1,0,0)=(\Delta_0^2|\Delta|)/x^2$. Since $Q_{\ell,n}(1,0,0)=|\Delta_0|$, we see that the above injection is actually the same as
 \begin{equation} \label{embedding 2}
     \mathcal{O}_{j_0} = \mathbb{Z}\left[\frac{\Delta_0 + \sqrt{\Delta_0}}{2} \right] \hookrightarrow \mathbb{Z}[\psi] + \ell^n \widetilde{R},  \hspace{0.5cm} \frac{\Delta_0 + \sqrt{\Delta_0}}{2} \mapsto \psi.
 \end{equation}
Recall that we also have embedding \eqref{injection}, which can be rewritten as
 \begin{equation} \label{embedding 3}
      \mathcal{O}_j=\mathbb{Z}\left[ \frac{\Delta + \sqrt{\Delta}}{2}\right] \hookrightarrow \mathbb{Z}[\psi] + \ell^n R.
 \end{equation}
 We remind the reader that the above injection \eqref{embedding 3} is again optimal, and that \eqref{embedding 1} is originally induced by \eqref{embedding 3}. It is then clear that the injections \eqref{embedding 1}, \eqref{embedding 2} and \eqref{embedding 3} are all compatible between each other, meaning that, after tensoring with $\mathbb{Q}$, one gets the same map $\iota: K \hookrightarrow \mathbb{B}_{\ell, \infty}$. In particular, $\mathcal{O}_j$ and $\mathcal{O}_{j_0}$ are contained inside the same imaginary quadratic field $K=\mathbb{Q}(\sqrt{\Delta})=\mathbb{Q}(\sqrt{\Delta_0})$.
 
 Consider now the order $\mathcal{O}:=\mathcal{O}_{j} + \mathcal{O}_{j_0} \subseteq K$. We have that $\iota(\mathcal{O}) \subseteq \mathbb{Z}[\psi] + \ell^n R$, and from the optimality of \eqref{embedding 3} it follows that $\mathcal{O}=\mathcal{O}_j$. Hence $\mathcal{O}_{j_0} \subseteq \mathcal{O}_j$, and this concludes the proof.
\end{proof}

Since at least one among $y$ and $z$ is non-zero, we also have that at least one among $\widetilde{y}:=y+ (\Delta_0 z)/2$ and $\widetilde{z}=z$ is non-zero. Note that $\widetilde{y} \in \frac{1}{2} \mathbb{Z}$ and $\widetilde{z} \in \mathbb{Z}$. Then we have
\begin{align*}
    \Delta_0^2 |\Delta|=\widetilde{Q}_{\ell,n} (\widetilde{x}, \widetilde{y}, \widetilde{z})=|\Delta_0|\widetilde{x}^2+4\ell^{2n+1} \widetilde{y}^2+\ell^{2n+1} |\Delta_0| \widetilde{z}^2 \geq \max \{4\ell^{2n+1} \widetilde{y}^2, \ell^{2n+1} |\Delta_0| \widetilde{z}^2 \} \geq \ell^{2n+1}
\end{align*}
 which implies
 \begin{equation} \label{inequality on n}
     n+1 \leq \frac{\log(\Delta_0^2 |\Delta|)}{2 \log \ell} + \frac{1}{2}.
 \end{equation}
Combining now \eqref{inequality on n} with Lemma \ref{Precise valuation} concludes the first case of the proof of Theorem \ref{Theorem valuation j-j0}.

\subsection{Second case: $\ell$ divides $\Delta$.}

For this part of the proof, we are going to heavily rely on \cite{Lauter_Viray_2015}, of which we have kept the notation. We again assume that the elliptic curve $E_0$ is given by an integral model over the ring of integers of $\mathbb{Q}(j_0)$ that has good reduction at $\mu \cap \mathbb{Q}(j_0)$.

Suppose initially that $\ell$ divides the conductor of the order $\mathcal{O}_j$. Let $H_{j} \subseteq F$ be a minimal extension of the ring class field $H_{j}$ such that there exists an elliptic curve $(E_j)_{/F}$ with $j(E_j)=j$ and having good reduction at all primes of $F$ lying above $\ell$. Fix such an elliptic curve $E_j$ and base-change it to the compositum $L=F\cdot H_{j_0}$. Consider also a prime $\mu_L \subseteq L$ lying above $\mu \subseteq H$ and denote by $A$ the ring of integers in the completion of the maximal unramified extension of $L_{\mu_L}$, with maximal ideal $\mu_L A \subseteq A$. By abuse of notation, we denote by $E_0, E_j$ the elliptic schemes over $A$ with generic fibers isomorphic to the base-changes of $E_0, E_j$ to the completion of the maximal unramified extension of $L_{\mu_L}$. The elliptic schemes ${E_j}$ and $E_0$ have good reduction over $A$ and, since $A$ is a complete discrete valuation ring of characteristic $0$ with algebraically closed residue field of characteristic $\ell>0$, we can use the same proof of Lemma \ref{Precise valuation} to see that
\begin{equation} \label{inequality valuation}
    v_{\mu} (j-j_0) \leq v_{\mu_L}(j-j_0) \leq \frac{d_0}{2} \cdot \max \{n \in \mathbb{N}_{\geq 1}: \Iso_{A/\mu_L^n A} (E_j, E_0) \neq \emptyset \}.
\end{equation}

Since $\ell \nmid \Delta_0$, we can now apply \cite[Proposition 4.1]{Lauter_Viray_2015} with $E=E_0$, $\mathcal{O}_{d_1}=\mathcal{O}_{j_0}$ and $\mathcal{O}_{d_2}= \mathcal{O}_j$. This proposition, used together with the fact that $\ell$ divides the conductor of $\mathcal{O}_j$, implies that $\Iso_{A/\mu_L^n A} (E_j, E_0)=\emptyset$ if $n>1$. Combined with \eqref{inequality valuation}, this gives
\[
    v_{\mu} (j-j_0) \leq \frac{d_0}{2}
\]
as desired. This yields the theorem in the case that $\ell$ divides the conductor of $\mathcal{O}_j$.

Assume now that $\ell$ divides $\Delta$ but does not divide the conductor of the order $\mathcal{O}_j$. Then, if again $E_j$ is an elliptic curve with $j(E_j)=j$, we can choose $F=H_j$ as a field where $E_j$ has a model with good reduction at all primes dividing $\ell$. This follows from \cite[Theorem 9]{Serre_Tate_1968}. If we complete $H$ at $\mu$, and we take $A$ to be the ring of integers in the completion of the maximal unramified extension of $H_{\mu}$ and $W$ to be the ring of integers in the completion of the maximal unramified extension of $\mathbb{Q}_\ell$, then $\text{Frac}(W) \subseteq \Frac(A)$ is a ramified degree $2$ field extension because the ramification index $e(\mu/\ell) = 2$ by our assumptions. Again by \cite[Proposition 4.1]{Lauter_Viray_2015}, since we are assuming that $\ell$ does not divide the conductor of $\mathcal{O}_j$, for every $n\in \mathbb{N}_{>0}$ we have
\begin{equation} \label{LV 4.1}
    \# \Iso_{A/\mu^n A}(E_0,E_j) \leq C \cdot \# S_n^{\text{Lie}}(E_0/A)
\end{equation}
where $C=C(j) \leq 6$ is a positive constant depending on $j$ and $S_n^{\text{Lie}}(E_0/A)$ is the set of all endomorphisms $\varphi \in \End_{A/\mu^n A}(E_0)$ satisfying the following three conditions (cfr. \cite[pag. 9218]{Lauter_Viray_2015}):
\begin{enumerate}
    \item $\varphi^2-\Delta \varphi+\frac{1}{4}(\Delta^2-\Delta)=0$;
    \item The inclusion $\mathbb{Z}[\varphi] \hookrightarrow \End_{A/\mu A}(E_0)$ is optimal at all primes $p\neq \ell$. We recall that an embedding of $\mathbb{Z}$-modules $\mathcal{O} \hookrightarrow R$ is\textit{ optimal at a prime $p$} if the equality 
    \[
    (\iota (\mathcal{O}) \otimes_\mathbb{Z} \mathbb{Q}_p) \cap (R \otimes_\mathbb{Z} \mathbb{Z}_p) = \iota (\mathcal{O}) \otimes_\mathbb{Z} \mathbb{Z}_p
    \]
    holds (note that the corresponding \cite[Definition 2.1]{Lauter_Viray_2015} contains a misprint);
    \item As endomorphism of $\text{Lie}(E_0 \text{ mod } \mu^n A)$ we have $\varphi \equiv \delta \text{ mod } \mu^n$, where $\delta \in A$ is a fixed root of the polynomial $x^2-\Delta x+\frac{1}{4}(\Delta^2-\Delta)$.
\end{enumerate}

The set $S_n^{\text{Lie}}(E_0/A)$ can be partitioned as
\[
S_n^{\text{Lie}}(E_0/A)= \bigcup_{m \in \mathbb{N}} S_{n,m}^{\text{Lie}}(E_0/A)
\]
where $S_{n,m}^{\text{Lie}}(E_0/A)$ consists of all the endomorphisms $\varphi \in S_n^{\text{Lie}}(E_0/A)$ such that
\[
\disc \left(\mathcal{O}_{j_0}[\varphi]\right)=m^2.
\]
We first claim that, under our assumptions, the sets $S_{n,0}^{\text{Lie}}(E_0/A)$ are empty for all $n \in \mathbb{N}_{>0}$. Indeed, let $\varphi \in S_{n,0}^{\text{Lie}}(E_0/A)$ so that $\disc \left(\mathcal{O}_{j_0}[\varphi]\right)=0$. Since a division quaternion algebra does not contain suborders of rank $3$, this in particular implies that $\mathcal{O}_{j_0}[\varphi]$ has rank $2$ as $\mathbb{Z}$-module, so that $\mathbb{Z}[\varphi]$ is isomorphic to an order in $K_{j_0}$, not necessarily contained in $\mathcal{O}_{j_0}$. By the definition of $S_n^{\text{Lie}}(E_0/A)$, the order $\mathbb{Z}[\varphi]$ has discriminant $\Delta$, and we deduce that $\mathbb{Z}[\varphi] \cong \mathcal{O}_j \subseteq K_{j_0}$. However, by assumption $\ell$ divides $\Delta$ but does not divide the conductor of $\mathcal{O}_j$. Hence $\ell$ must divide the discriminant of $K_{j_0}$ which in turn implies $\ell \mid \Delta_0$, contradicting our hypotheses. This proves the claim.

On the other hand, in the second paragraph of \cite[pag. 9247]{Lauter_Viray_2015} it is proved that, when $\ell$ divides $\Delta$ but does not divide the conductor of $\mathcal{O}_{j}$, and $\ell \nmid \Delta_0$, then for every $m>0$ and $n>1$, the set $S_{n,m}^{\text{Lie}}(E/A)$ is empty. We deduce that $S_n^{\text{Lie}}(E/A)=\emptyset$ for all $n>1$, and combining this with inequality \eqref{LV 4.1} we obtain $\Iso_{A/\mu^n A}(E_0,E_j)=\emptyset$ for all $n>1$. Finally, using \cite[Proposition 2.3]{Gross_Zagier_1985} we obtain
\[
v_\mu(j-j_0)= \frac{1}{2} \# \Iso_{A/\mu A}(E_0,E_j) \leq \frac{d_0}{2}
\]
and this concludes the proof of Theorem \ref{Theorem valuation j-j0}.

\section{Proof of Theorem \ref{main theorem j-j0}} \label{sec: proof of main theorem j-j0}

The main scope of this section is to present the proof of Theorem \ref{main theorem j-j0}. At the end of this proof, we will point at the precise estimates that can be used to prove Theorem \ref{main theorem j+3375} and similar results, and we will provide a proof of the fact (stated in the introduction) that the extension $\mathbb{Q} \subseteq \mathbb{Q}(j_0)$ can be Galois for at most a finite number of singular moduli $j_0$. Before starting, let us recall some notation already used in the introduction. For a number field $K$ we denote by $\mathcal{M}_K$ the set of all places of $K$ and by $\mathcal{M}^{\infty}_K \subseteq \mathcal{M}_K$ the subset of all the infinite ones. For every $w \in \mathcal{M}_K \setminus \mathcal{M}^{\infty}_K$ we indicate by $|\cdot |_w$ the absolute value in the class of $w$ normalized as follows: if $\mathfrak{p}_w$ denotes the prime ideal corresponding to $w$ and $p_w$ is the rational prime lying below $\mathfrak{p}_w$, then 
\[
|x|_w=p_w^{-v_{\mathfrak{p}_w}(x)/e_w}
\]
for all $x \in K\setminus \{ 0 \}$, where $v_{\mathfrak{p}_w}(x)$ is the exponent with which the prime $\mathfrak{p}_w$ appears in the factorization of $x$, and $e_w$ is the ramification index of $\mathfrak{p}_w$ over $p_w$.

\begin{proof}[Proof of Theorem \ref{main theorem j-j0}]

Let $(j_0,S)$ be a nice $\Delta_0$-pair with $\Delta_0 < -4$ and $\#S \leq 2$. We can assume without loss of generality that $\#S=2$, since if $S$ contains fewer than two elements the statement of the theorem becomes weaker. Hence we can write $S=\{\ell_1, \ell_2 \}$ with $\ell_1, \ell_2 \in \mathbb{N}$ two distinct primes.

In order to prove Theorem \ref{main theorem j-j0}, we follow the strategy used in \cite{BHK_2018} to prove the emptiness of the set of singular units. Let  $j$ be a singular modulus of discriminant $\Delta$ such that $j-j_0$ is an $S$-unit, and let $h(\cdot)$ denote the logarithmic Weil height on algebraic numbers. By the usual properties of height functions \cite[Lemma 1.5.18]{Bombieri_Gubler_book_2006}, we have
\begin{equation} \label{Weil height}
   h(j-j_0)=h((j-j_0)^{-1})=\frac{1}{[\mathbb{Q}(j-j_0):\mathbb{Q}]} \sum_{v \in \mathcal{M}_{\mathbb{Q}(j-j_0)}} d_v \log^+ |(j-j_0)^{-1}|_v = A + N 
\end{equation}
where $d_v:=[\mathbb{Q}(j-j_0)_v:\mathbb{Q}_v]$ is the local degree of the field $\mathbb{Q}(j-j_0)$ at the place $v$ and
\begin{align*}
    &A:=\frac{1}{[\mathbb{Q}(j-j_0):\mathbb{Q}]} \sum_{v \in \mathcal{M}^{\infty}_{\mathbb{Q}(j-j_0)}} d_v \log^+ |(j-j_0)^{-1}|_v, \\ &N:=\frac{1}{[\mathbb{Q}(j-j_0):\mathbb{Q}]} \sum_{v \mid \ell_1 \ell_2 } d_v \log |(j-j_0)^{-1}|_v 
\end{align*}
are, respectively, the archimedean and non-archimedean components of the height. Notice that the expression for $N$ follows from our assumption on $j-j_0$ being an $S$-unit and from the fact that $j-j_0$ is an algebraic integer. We study these two components separately, starting with the archimedean one. From now on, we assume $|\Delta| > \max \{|\Delta_0|, 10^{15} \}$.

Denote by $C_0$ and $C_\Delta$ the class numbers of the orders  associated to $j_0$ and to $j$ respectively. Then by \cite[Corollary 4.2 (1)]{Cai_2020} we have
\begin{equation} \label{archimedean part estimate 1}
    A \leq \frac{8 F \log |\Delta| \cdot C_0}{[\mathbb{Q}(j-j_0): \mathbb{Q}]} + \log \left( \frac{F \log |\Delta| \cdot C_0 \cdot |\Delta|^{1/2}}{[\mathbb{Q}(j-j_0): \mathbb{Q}]} \right) + 4\log |\Delta_0| + 0.33
\end{equation}
where $F:= \max \{2^{\omega(a)}: a\leq |\Delta|^{1/2} \}$ and $\omega(n)$ denotes the number of prime divisors of an integer $n\in \mathbb{N}$. Using \cite[Theorem 4.1]{Faye_Riffaut_2018} we have
\[
    [\mathbb{Q}(j-j_0): \mathbb{Q}] = [\mathbb{Q}(j,j_0): \mathbb{Q}] \geq [\mathbb{Q}(j): \mathbb{Q}]=C_\Delta
\]
which, combined with \eqref{archimedean part estimate 1}, gives
\begin{equation} \label{archimedean part estimate 2}
    A \leq \frac{8 F \log |\Delta| \cdot C_0}{C_\Delta} + \log \left( \frac{F \log |\Delta| \cdot C_0 \cdot |\Delta|^{1/2}}{C_\Delta} \right) + 4\log |\Delta_0| + 0.33.
\end{equation}

As far as the non-archimedean part is concerned, we have

\begin{align} \label{non archimedean part estimate 1}
N&=\frac{1}{[\mathbb{Q}(j-j_0):\mathbb{Q}]} \sum_{v \mid \ell_1 \ell_2 } d_v \log |(j-j_0)^{-1}|_v  
= \frac{1}{[\mathbb{Q}(j-j_0):\mathbb{Q}]} \sum_{i\in \{1,2\}} \sum_{\mathfrak{p} \mid \ell_i} v_{\mathfrak{p}}(j-j_0) \log \ell_i^{f_{\mathfrak{p}}} \\
&=\frac{\log \ell_1}{[\Q(j-j_0):\Q]} \sum_{\mathfrak{p} \mid \ell_1} v_{\mathfrak{p}}(j-j_0) f_{\mathfrak{p}} + \frac{\log \ell_2}{[\Q(j-j_0):\Q]} \sum_{\mathfrak{p} \mid  \ell_2} v_{\mathfrak{p}}(j-j_0) f_{\mathfrak{p}} \nonumber 
\end{align} 
where $f_{\mathfrak{p}}$ denotes the residue degree of the prime $\mathfrak{p} \subseteq \Q(j-j_0)$ lying over $\mathfrak{p}\cap \mathbb{Q}$. For every $\mathfrak{p}\mid \ell_1 \ell_2$, we choose a prime ideal $\mu \subseteq H$ that divides $\mathfrak{p}$, where $H$ denotes the compositum inside $\overline{\mathbb{Q}}$ of the ring class fields relative to $j$ and $j_0$. Note that this makes sense, since we have $\mathbb{Q}(j-j_0)\subseteq \mathbb{Q}(j,j_0) \subseteq H$ (the first inclusion is actually an equality by \cite[Theorem 4.1]{Faye_Riffaut_2018}). We wish now to use Theorem \ref{Theorem valuation j-j0} to bound $v_\mu(j-j_0)$ for all these primes $\mu$. Let's check that the hypotheses of the theorem are verified in our context:
\begin{itemize}
    \item since we are assuming $|\Delta_0| < |\Delta|$, certainly we have $\Delta \neq \Delta_0$;
    \item since $(j_0,S)$ is a nice $\Delta_0$-pair, for $i \in \{1,2\}$ the prime $\ell_i$ splits completely in $\mathbb{Q}(j_0)$. In particular, $\mu \cap \mathbb{Q}(j_0)$ has residue degree $1$, as required;
    \item since $(j_0,S)$ is a nice $\Delta_0$-pair, for $i \in \{1,2\}$ the prime $\ell_i$ does not divide either $\Delta_0$ or $N_{\mathbb{Q}(j_0)/\mathbb{Q}}(j_0(j_0-1728))$. In particular, this last condition implies that the elliptic curve 
    \[
    {E_0}_{/\mathbb{Q}(j_0)}: y^2+xy=x^3-\frac{36}{j_0-1728} x - \frac{1}{j_0-1728}
    \]
    with $j(E_0)=j_0$, has good reduction at $\mu$.
\end{itemize}
This discussion shows that we can apply Theorem \ref{Theorem valuation j-j0} to bound $v_\mu(j-j_0)$. Notice that under our assumptions we have, in the notation of the theorem, that $d_0=2$ since $\ell_i \nmid N_{\mathbb{Q}(j_0)/\mathbb{Q}}(j_0(j_0-1728))$ for $i \in \{1,2\}$. Moreover, the imaginary quadratic order associated to $j$ cannot contain the order associated to $j_0$ because $|\Delta| >|\Delta_0|$. Thus we obtain
\[
    v_{\mathfrak{p}}(j-j_0) \leq v_{\mu}(j-j_0) \leq \max \left \{  \frac{\log(\Delta_0^2|\Delta|)}{2 \log \ell_i}+ \frac{1}{2}, 1 \right \} 
\]

for all primes $\mathfrak{p} \mid \ell_i$. Combining this with \eqref{non archimedean part estimate 1} and setting $L:=\max \{ \ell_1, \ell_2 \}$ we obtain
\begin{align} \label{non archimedean part estimate 2}
    N &\leq \frac{\log \ell_1}{[\Q(j-j_0):\Q]} \sum_{\mathfrak{p} \mid \ell_1} \max \left \{  \frac{\log(\Delta_0^2|\Delta|)}{2 \log \ell_1}+ \frac{1}{2}, 1 \right \} f_{\mathfrak{p}} + \frac{\log \ell_2}{[\Q(j-j_0):\Q]} \sum_{\mathfrak{p} \mid  \ell_2} \max \left \{  \frac{\log(\Delta_0^2|\Delta|)}{2 \log \ell_2}+ \frac{1}{2}, 1 \right \} f_{\mathfrak{p}} \\
    &\leq (\log \ell_1) \max \left \{  \frac{\log(\Delta_0^2|\Delta|)}{2 \log \ell_1}+ \frac{1}{2}, 1 \right \} + (\log \ell_2) \max \left \{  \frac{\log(\Delta_0^2|\Delta|)}{2 \log \ell_2}+ \frac{1}{2}, 1 \right \}  \nonumber \\
    &= \max \left \{  \frac{\log(\Delta_0^2|\Delta|)}{2}+ \frac{\log \ell_1}{2}, \log \ell_1 \right \} + \max \left \{  \frac{\log(\Delta_0^2|\Delta|)}{2}+ \frac{\log \ell_2}{2}, \log \ell_2 \right \} \nonumber \\ 
    &\leq 2 \max \left \{  \frac{\log(\Delta_0^2|\Delta|)}{2}+ \frac{\log L}{2}, \log L \right \} = \max \{\log(\Delta_0^2|\Delta|) + \log L, 2 \log L \} \nonumber 
\end{align}
where in the second inequality we have used the fact that, for every number field $K$ and any prime $q \in \mathbb{N}$, we always have $\sum_{\mathfrak{q}\mid q} f_{\mathfrak{q}} \leq [K:\mathbb{Q}]$ (here the sum is taken over the prime ideals of $K$ lying above $q$).
Using now together \eqref{Weil height}, \eqref{archimedean part estimate 2} and \eqref{non archimedean part estimate 2} we obtain the following upper bound
\begin{align} \label{upper bound for h(j-j0)}
    h(j-j_0) \leq \frac{8 F \log |\Delta| \cdot C_0}{C_\Delta} &+ \log \left( \frac{F \log |\Delta| \cdot C_0 \cdot |\Delta|^{1/2}}{C_\Delta} \right) + 4\log |\Delta_0| + 0.33  \\
    &+ \max \{\log(\Delta_0^2|\Delta|) + \log L, 2 \log L \} \nonumber
\end{align}
for the Weil height of $j-j_0$. We now look into lower bounds.

In order to find a lower bound for $h(j-j_0)$, we first reduce to the problem of finding a lower bound for $h(j)$ by means of the elementary inequality
\begin{equation} \label{eq:height_inequality_elementary}
    h(j-j_0) \geq h(j)-h(j_0)-\log 2
\end{equation}
see \cite[Proposition 1.5.15]{Bombieri_Gubler_book_2006}. As for bounding $h(j)$, we use the lower bound \cite[Proposition 4.3]{BHK_2018} 
\begin{equation} \label{The hard bound in BHK}
    h(j) \geq \frac{3}{\sqrt{5}} \log |\Delta| - 9.79
\end{equation}
together with \cite[Proposition 4.1]{BHK_2018}
\begin{equation} \label{The easy bound in BHK}
     h(j) \geq \frac{\pi |\Delta|^{1/2}-0.01}{C_\Delta} 
\end{equation}
which generally holds for $|\Delta| \geq 16$. Combining \eqref{eq:height_inequality_elementary} with \eqref{The hard bound in BHK} and \eqref{The easy bound in BHK}, and adding $1$ on both sides, we obtain
\begin{equation} \label{lower bound for h(j-j0)}
    Y(\Delta):=\max \left\{ \frac{3}{\sqrt{5}} \log |\Delta| - 8.79, \frac{\pi |\Delta|^{1/2}}{C_\Delta}\right\} \leq h(j-j_0)+h(j_0)+\log 2+1.
\end{equation}
Concatenating now \eqref{lower bound for h(j-j0)} with \eqref{upper bound for h(j-j0)}, and dividing both sides by $Y(\Delta)$, yields the inequality
\begin{equation} \label{1 leq A+B+C+D}
    1 \leq A(\Delta) + B(\Delta) + C(\Delta) + D(\Delta)
\end{equation}
where 
\begin{align*}
    &A(\Delta)= \frac{8 F \log |\Delta| \cdot C_0}{Y(\Delta) C_\Delta},\\
    &B(\Delta)= \frac{\log \left( F \log |\Delta| \right) + \log C_0 + 4\log |\Delta_0| + h(j_0) + 1.33 + \log 2}{Y(\Delta)}, \\
    &C(\Delta)= \frac{1}{Y(\Delta)} \log \left( \frac{|\Delta|^{1/2}}{C_\Delta} \right) \\
    &D(\Delta)= \frac{1}{Y(\Delta)} \cdot \max \{\log(\Delta_0^2|\Delta|) + \log L, 2 \log L \}.
\end{align*}
We want to show that \eqref{1 leq A+B+C+D} cannot hold if $|\Delta|$ is sufficiently large. As far as estimating the first three terms of \eqref{1 leq A+B+C+D} is concerned, we find ourselves in the same situation as Cai in \cite[Sections 6.1-6.4]{Cai_2020}, and we can directly use the bounds therein obtained. More precisely from \cite[Section 6.2]{Cai_2020} we have, since $|\Delta| > 10^{15}$, that
\begin{equation*} 
    A(\Delta) \leq \frac{8 F \log |\Delta| \cdot C_0}{\pi |\Delta|^{1/2}} \leq \frac{8C_0}{\pi} |\Delta|^{-0.1908} 
\end{equation*}
so for every $\varepsilon_A>0$,
\begin{equation} \label{estimate for A(delta)}
     A(\Delta) \leq \frac{8C_0}{\pi} |\Delta|^{-0.1908} < \varepsilon_A
\end{equation}
holds for $|\Delta|$ sufficiently large. Moreover, using
\[
\log (F \log |\Delta|) \leq \frac{\log 2}{2} \cdot \frac{\log |\Delta|}{\log \log |\Delta| - c_1-\log 2} + \log \log |\Delta|
\]

which is \cite[Inequality (5.8)]{BHK_2018} (here $c_1 \in \mathbb{R}$ is an effectively computable absolute constant defined in \cite[Section 5.2]{BHK_2018}), we have that, for every $\varepsilon_B>0$, the inequality
\begin{equation} \label{estimate for B(delta)}
    B(\Delta) \leq \frac{1}{(3/\sqrt{5})\log |\Delta| -8.79} \left(\frac{\log 2}{2} \cdot \frac{\log |\Delta|}{\log \log |\Delta| - c_1-\log 2} + \log \log |\Delta| + K \right) < \varepsilon_B
\end{equation}
where $K:=\log C_0 + 4\log |\Delta_0| + h(j_0) + 1.33 + \log 2$, holds for $|\Delta|$ sufficiently large. Finally, using the fact that $x \mapsto \log(x)/x$ is a decreasing function when $x\geq 4$, for every $\varepsilon_C>0$ one has

\begin{equation} \label{estimate for C(delta)}
    C(\Delta) \leq \frac{1}{Y(\Delta)} \log \left( \pi^{-1} Y(\Delta)\right) \leq \frac{1}{(3/\sqrt{5})\log |\Delta| -8.79} \log \left(\pi^{-1} \left(\frac{3}{\sqrt{5}}\log |\Delta| -8.79 \right) \right) < \varepsilon_C
\end{equation}
for $|\Delta|$ sufficiently large. We are then left with bounding $D(\Delta)$ from above. For $|\Delta| \geq L/|\Delta_0|^2$ we have
\begin{align*}
    D(\Delta)&= \frac{1}{Y(\Delta)} \cdot   \max \{\log(\Delta_0^2|\Delta|) + \log L, 2 \log L \}  \\
    &\leq \frac{1}{\frac{3}{\sqrt{5}}\log |\Delta| -8.79} \cdot \left(\log |\Delta|+\log L \left( \frac{\log \Delta_0^2}{\log L}+1 \right) \right) \\
    &= \frac{\sqrt{5}}{3} + \frac{1}{\frac{3}{\sqrt{5}}\log |\Delta| -8.79} \cdot \left(\frac{\sqrt{5}}{3}\cdot 8.79 + \log L \left( \frac{\log \Delta_0^2}{\log L}+1 \right) \right)
\end{align*}
 so for every $\varepsilon_D>0$ we obtain
\begin{equation} \label{estimate for D(delta)}
    D(\Delta) \leq \frac{\sqrt{5}}{3} + \varepsilon_D \leq 0.75 + \varepsilon_D
\end{equation}
for $|\Delta|$ sufficiently large (depending on $\Delta_0$ and $\ell_1, \ell_2$). We can now combine \eqref{estimate for A(delta)}, \eqref{estimate for B(delta)}, \eqref{estimate for C(delta)}, \eqref{estimate for D(delta)} with \eqref{1 leq A+B+C+D} to obtain
\begin{equation} \label{1<0.75 + epsilon}
    1 \leq \varepsilon_A + \varepsilon_B + \varepsilon_C  + \varepsilon_D + 0.75 
\end{equation}
which holds for $|\Delta| \gg_{\ell_1, \ell_2, \Delta_0, \varepsilon_A, \varepsilon_B, \varepsilon_C, \varepsilon_D} 0$. Choosing $\varepsilon_A, \varepsilon_B, \varepsilon_C, \varepsilon_D$ small enough, the inequality cannot be verified for sufficiently large $|\Delta|$. This proves that there are at most finitely many singular moduli $j$ such that $j-j_0$ is an $S$-unit, and concludes the proof of the first part of Theorem \ref{main theorem j-j0}.

We now begin the proof of the second part of Theorem \ref{main theorem j-j0}. Suppose $\mathbb{Q} \subseteq \mathbb{Q}(j_0)$ is not Galois. We first claim that every prime in $S$ must be split in $\mathbb{Q}(\sqrt{\Delta_0})$.
Indeed, assume by contradiction that a prime $\ell \in S$ is inert in $\mathbb{Q}(\sqrt{\Delta_0})$ (it cannot ramify by definition of a nice $\Delta_0$-pair).  Let $H_\mathcal{O}:= \mathbb{Q}(j_0, \sqrt{\Delta_0})$ which is a semidihedral Galois extension of $\mathbb{Q}$, and let 
 \[
 H:= \Gal(H_\mathcal{O}/\mathbb{Q}(j_0)) \subseteq \Gal(H_\mathcal{O}/\mathbb{Q})=:G
 \]
 
with generator $\sigma \in H$. Since $\ell$ splits completely in $\mathbb{Q}(j_0)$ and is inert in $\mathbb{Q}(\sqrt{\Delta_0})$, the decomposition group of any prime of $H_\mathcal{O}$ above $\ell$ has order $2$ and certainly contains $H$, since all the primes in $\mathbb{Q}(j_0)$ lying above $\ell$ are inert in $\mathbb{Q}(j_0) \subseteq H_\mathcal{O}$. Hence, every such decomposition group must be equal to $H$ and this means in particular that $\tau H \tau^{-1} = H$ for all $\tau \in G$. We deduce that $\sigma$ commutes with every element of $G$, so $G$ must be abelian and we reach the desired contradiction.
 
 Let now $j \in \overline{\mathbb{Q}}$ be a singular modulus of discriminant $\Delta$ such that $j-j_0$ is an $S$-unit. Since $j-j_0$ cannot be a unit by \cite[Corollary 1.3]{Li_2018}, there exists a prime $\ell \in S$ dividing the norm of $j-j_0$. This implies that there exists a number field $K$, a prime $\mu \subseteq K$ lying above $\ell$ and two elliptic curves $E_0,E_j$ defined over $K$ with good reduction at $\mu$ such that $j(E_j)=j$, $j(E_0)=j_0$ and $E_0 \text{ mod } \mu \cong_{\overline{\mathbb{F}}_\ell} E_j \text{ mod } \mu$. Moreover, since $\ell$ splits in $\mathbb{Q}(\sqrt{\Delta_0})$ by the discussion above, both $E_0$ and $E_j$ have ordinary reduction modulo $\mu$ by \cite[Chapter 13, Theorem 12]{Lang_1987}. From the reduction theory of CM orders (see again \cite[Chapter 13, Theorem 12]{Lang_1987}) and the fact that $\ell \nmid \Delta_0$, we deduce that $\Delta=\ell^{2n} \Delta_0$ for some non-negative integer $n$, as wanted. 
 \end{proof}
 
\begin{remark}
The proof of Theorem \ref{main theorem j-j0} can now be specialized to different situations to obtain explicit results on singular differences that are $S$-units. Indeed, for a given nice $\Delta_0$-pair $(j_0,S)$, it suffices to find a discriminant $\Delta$ whose absolute value is sufficiently large to violate inequality \eqref{1<0.75 + epsilon}, which in turn is a combination of the explicit inequalities \eqref{estimate for A(delta)}, \eqref{estimate for B(delta)}, \eqref{estimate for C(delta)} and \eqref{estimate for D(delta)}. For instance, in the case of Theorem \ref{main theorem j+3375}, with the choice of the nice $(-7)$-pair $(-3375,\{13,17\})$ and $|\Delta|>10^{81}$ one gets
\[
\varepsilon_A + \varepsilon_B + \varepsilon_C  + \varepsilon_D < 0.2485
\]
and similarly with the other choices of primes in the theorem.
\end{remark}

We conclude this section by proving that the hypothesis on the extension $\mathbb{Q} \subseteq \mathbb{Q}(j_0)$ being Galois, which appears in the statement of Theorem \ref{main theorem j-j0}, is verified only for finitely many singular moduli $j_0$.

\begin{proposition} \label{prop:Galois_for_finitely_many_j}
There are at most finitely many singular moduli $j \in \overline{\mathbb{Q}}$ such that the extension $\mathbb{Q} \subseteq \mathbb{Q}(j)$ is Galois.
\end{proposition}

\begin{proof}
Let $\mathcal{O}$ be the imaginary quadratic order relative to $j$, and denote by $K$ its fraction field, with ring of integers $\mathcal{O}_K$ and discriminant $\Delta_K$. By \cite[Corollary 3.3]{Allombert_Bilu_Pizarro_2015} the extension $\mathbb{Q} \subseteq \mathbb{Q}(j)$ is Galois if and only if the class group $\mathrm{Pic}(\mathcal{O})$ of the order $\mathcal{O}$ is elementary $2$-abelian. Since the natural homomorphism $\mathrm{Pic}(\mathcal{O}) \to \mathrm{Pic}(\mathcal{O}_K)$ is surjective (see \cite[Proposition I.12.9]{Neukirch_book_1999}), we deduce that also the class group of $\mathcal{O}_K$ has exponent $2$. Hence, by \cite[Theorem 1]{Weinberger_1973}, there exists a negative fundamental discriminant $D$ such that either $|\Delta_K| \leq 5460$ or $\Delta_K=D$. In particular, we have only a finite number of possibilities for the imaginary quadratic field $K$. 

In order to show that there is also a finite number of possibilities for the order $\mathcal{O}$, we need to bound the possible conductors $f:=|\mathcal{O}_K : \mathcal{O}|$. To do so, consider the genus field $G_\mathcal{O}$ relative to the order $\mathcal{O}$: it is the maximal subextension of the ring class field $H_\mathcal{O}$ that contains $K$ and is abelian over $\mathbb{Q}$. It can also be described as the fixed field by $\mathrm{Pic}(\mathcal{O})^2$ under the Artin isomorphism $\Gal(H_\mathcal{O}/K) \cong \mathrm{Pic}(\mathcal{O})$, see \cite{Halter-Koch_1971} and \cite[Section 2.2]{Kuhne_2013} (note also that the first part of the proof of \cite[Theorem 6.1]{Cox_book_2013} carries over to non-maximal orders after making appropriate modifications). Since $\mathrm{Pic}(\mathcal{O})$ is elementary $2$-abelian by hypothesis, we deduce that $G_\mathcal{O}=H_\mathcal{O}$ and then \cite[Equation (2.3)]{Kuhne_2013} implies that
\begin{equation} \label{eq:inequality_class_group}
    \# \mathrm{Pic}(\mathcal{O}) \leq 2^{\omega (f^2 \Delta_K)+1}
\end{equation}
where $\omega (f^2 \Delta_K)$ denotes the number of distinct prime divisors of $f^2 \Delta_K$. On the other hand, by \cite[Theorem 7.24]{Cox_book_2013} we can write
\begin{equation} \label{eq:class_number_formula}
    f \prod_{p\mid f} \left( 1-\left(\frac{\Delta_K}{p} \right) \frac{1}{p} \right) = \frac{|\mathcal{O}_K^\times : \mathcal{O}^\times|}{\# \mathrm{Pic}(\mathcal{O}_K)} \# \mathrm{Pic}(\mathcal{O}).
\end{equation}
Combining \eqref{eq:inequality_class_group} with \eqref{eq:class_number_formula}, and using the fact that $K$ ranges among a finite set of imaginary quadratic fields, it is not difficult to see that $f$ must be in fact bounded. This concludes the proof.
\end{proof}

\section{Theorems \ref{main theorem j-1728} and \ref{main theorem j-0}} \label{sec: proof the other two theorems}

The proofs of Theorems \ref{main theorem j-1728} and \ref{main theorem j-0}, after a preliminary reduction step, become analogous to the proof of Theorem \ref{main theorem j-j0}. The reader may then wonder why we decided to not write one single argument for all these results. The reason is double: first, for clarity of exposition, since already the proof of Theorem \ref{main theorem j-j0} contains quite involved computations. Second, because differences of the form $j-j_0$ with $j_0 \in \{0,1728\}$ require some extra attention due to the fact that the corresponding elliptic curves $(E_0)_{/\mathbb{Q}}$ with $j(E_0)=j_0$ have more geometric automorphisms than in the other cases. After pondering all these aspects, we chose to only sketch the proofs of the two aforementioned theorems, outlining with all the details only the parts in which they differ from the proof of Theorem \ref{main theorem j-j0}. We begin with Theorem \ref{main theorem j-1728}.

\begin{proof}[Proof of Theorem \ref{main theorem j-1728}]

First of all, we show that it is sufficient to prove that, under the assumptions of the theorem, the set of singular moduli $j$ such that $j-1728$ is an $\{\ell \}$-unit is finite and the discriminants of its elements can be effectively bounded. Indeed, suppose that $j-1728$ is a singular $S_\ell$-unit and assume that $p \in S_0$ is a prime dividing its norm $N_{\mathbb{Q}(j)/\mathbb{Q}}(j-1728)$. Then for every prime $\mathfrak{p} \subseteq \overline{\mathbb{Q}}$ lying above $p$ we have $j\equiv 1728 \text{ mod } \mathfrak{p}$ and $\mathfrak{p}$ is a prime of ordinary reduction for every elliptic curve over $\overline{\mathbb{Q}}$ with $j$-invariant $1728$ or $j$. In particular, $1728$ and $j$ must be associated with the same imaginary quadratic field $\mathbb{Q}(\sqrt{-1})$. It has then been proved in \cite[Claim 6.1]{Campagna_2020} that in this case, there are at least other $3$ primes not congruent to $1$ modulo $4$ dividing this norm. In particular, $j-1728$ cannot be a singular $S_\ell$-unit (the existence of this argument is also remarked in \cite[Section 1.1]{Herrero_Menares_Rivera_Part3}).
 
Hence we are reduced to bounding the discriminants of the singular moduli $j$ such that $j-1728$ is an $\{\ell\}$-unit for $\ell \geq 5$ a prime congruent to $3$ modulo $4$.
Let then $j \in \overline{\mathbb{Q}}$ be a singular modulus such that $j-1728$ is an $\{\ell\}$-unit. In the same way as in the previous section, we compute the Weil height
\begin{equation} \label{Weil height j-1728}
   h(j-1728)=h((j-1728)^{-1})=\frac{1}{[\mathbb{Q}(j):\mathbb{Q}]} \sum_{v \in \mathcal{M}_{\mathbb{Q}(j)}} d_v \log^+ |(j-1728)^{-1}|_v = A + N 
\end{equation}
where, again, $d_v:=[\mathbb{Q}(j)_v:\mathbb{Q}_v]$ is the local degree at the place $v$ and
\[
A:=\frac{1}{[\mathbb{Q}(j):\mathbb{Q}]} \sum_{v \in \mathcal{M}^{\infty}_{\mathbb{Q}(j)}} d_v \log^+ |(j-1728)^{-1}|_v \hspace{0.5cm} \text{and} \hspace{0.5cm} N:=\frac{1}{[\mathbb{Q}(j):\mathbb{Q}]} \sum_{v \mid \ell } d_v \log^+ |(j-1728)^{-1}|_v
\]
are, respectively, the archimedean and non-archimedean components of the height. For $|\Delta|$ big enough, we can bound the archimedean component using another time the work of Cai \cite{Cai_2020}. More precisely, \cite[Corollary 4.2]{Cai_2020} gives for $|\Delta| \geq 10^{14}$
\begin{equation} \label{Archimedean j-1728}
A \leq \frac{4 F \log |\Delta|}{C_{\Delta}}+2 \log \frac{F |\Delta|^{1/2} \log |\Delta|}{C_{\Delta}}-2.68
\end{equation}
where $C_{\Delta}$ is the class number of the order of discriminant $\Delta$ and $F= \max \{2^{\omega(a)}: a\leq |\Delta|^{1/2} \}$ as in the previous section. The non-archimedean component can be rewritten as
\begin{align} \label{non archimedean computation j-1728}
N= \frac{1}{[\Q(j):\Q]} \sum_{\mathfrak{p} \mid \ell} v_{\mathfrak{p}}(j-1728) \log \ell^{f_{\mathfrak{p}}}=\frac{\log \ell}{[\Q(j):\Q]} \sum_{\mathfrak{p} \mid \ell} v_{\mathfrak{p}}(j-1728) f_{\mathfrak{p}}
\end{align} 
where the sum runs over primes $\mathfrak{p}$ of $\mathbb{Q}(j)$ lying above $\ell$, and $f_\mathfrak{p}$ denotes the residue degree of $\mathfrak{p}$ over $\ell$. To estimate the valuation from above, we can apply Theorem \ref{Theorem valuation j-j0} since all the hypotheses are met also in this case: $\ell$ has certainly degree $1$ in $\mathbb{Q}(1728)=\mathbb{Q}$ and is coprime with $-4=\disc \mathbb{Q}(i)$. Moreover, the elliptic curve ${E_{1728}}_{/\mathbb{Q}}: y^2=x^3+x$ has $j(E_{1728})=1728$ and good reduction at all primes $\ell \neq 2$. We deduce that for all $\mathfrak{p} \mid \ell$ we have
\[
v_{\mathfrak{p}}(j-1728) \leq \max \left \{ \frac{\log(16|\Delta|)}{\log \ell} + 1, 2 \right \}
\]
where in the application of Theorem \ref{Theorem valuation j-j0} one has $d_0=4$ since $\ell \geq 5$ (see \cite[III, Theorem 10.1]{Silverman_2009}). Combining the above estimate with \eqref{non archimedean computation j-1728} we obtain
\begin{equation} \label{non-archimedean j-1728}
    N \leq \max \{\log(16|\Delta|) + \log \ell, 2\log \ell \}
\end{equation}
so putting together \eqref{Weil height j-1728}, \eqref{Archimedean j-1728} and \eqref{non-archimedean j-1728} we get
\begin{equation} \label{A+N j-1728}
   h(j-1728) \leq \frac{4 F \log |\Delta|}{C_{\Delta}}+2 \log \frac{F |\Delta|^{1/2} \log |\Delta|}{C_{\Delta}}-2.68 + \max \{\log (16 |\Delta|) + \log \ell, 2\log \ell \}
\end{equation}
for $|\Delta| \geq 10^{14}$. Now the lower bound \eqref{lower bound for h(j-j0)} allows to conclude exactly in the same way as in the proof of Theorem \ref{main theorem j-j0}.
\end{proof}

As the reader may have noticed, the intimate reason why the proofs of Theorems \ref{main theorem j-j0} and \ref{main theorem j-1728} work out is that the lower bound \eqref{lower bound for h(j-j0)} is sufficiently good to prevail on the estimates \eqref{non archimedean part estimate 2} and \eqref{non-archimedean j-1728} for the non-archimedean parts of the relevant Weil heights. This will not be the case for $j_0=0$, since in this case one has to take $d_0 \geq 6$ in the inequalities of Theorem \ref{Theorem valuation j-j0}. This is the reason why the proof of Theorem \ref{main theorem j-0} is conditional under GRH. However, as already mentioned in the introduction, Theorem \ref{main theorem j-0} does not need the full strength of the Generalized Riemann Hypothesis to be proved, but only that a weaker condition on the Dirichlet $L$-functions associated to imaginary quadratic fields holds. The goal of the subsequent discussion is to introduce this condition, and to deduce from its assumption a lower bound for the Weil height of a singular modulus that is sharp enough to prove Theorem \ref{main theorem j-0} with our methods.

Recall that non-principal real primitive Dirichlet characters are precisely the Kronecker symbols attached to quadratic field extensions of $\mathbb{Q}$. We say that such a Dirichlet character has discriminant $D \in \mathbb{Z}$ if it is the Kronecker symbol attached to a quadratic field of discriminant $D$.

\begin{definition} \label{definition properti p(k)}
Let $k \in \mathbb{R}$ be a non-negative real number. A non-principal real primitive Dirichlet character $\chi$ of discriminant $D$ is said to satisfy \textit{property $P(k)$} if 
\[
\frac{L'(\chi, 1)}{L(\chi,1)} \geq -0.2485 \log |D| - k
\]
where the left-hand side of the inequality is the logarithmic derivative of the Dirichlet $L$-function $L(\chi,s)$ associated to $\chi$.
\end{definition}

\begin{remark}
The inequality appearing in Definition \ref{definition properti p(k)} may seem a bit arbitrary, and indeed it is. Actually for our purposes, we could take any inequality of the form
\[
\frac{L'(\chi, 1)}{L(\chi,1)} \geq -c \log |D| - k
\]
with $c<0.25$ as a definition for the property $P(k)$, and all the following proofs would work in the same way.
\end{remark}

\begin{remark}
It is proved in \cite{Mourtada_Murty_2013} that the logarithmic derivative of Dirichlet $L$-functions attached to Kronecker symbols of imaginary quadratic fields is actually positive for infinitely many negative fundamental discriminants. In particular, property $P(0)$ holds for infinitely many real primitive Dirichlet characters of negative discriminant.
\end{remark}

Let now $j \in \overline{\mathbb{Q}}$ be a singular modulus relative to an order in the imaginary quadratic field $K$. Under the assumption that the Kronecker symbol associated to $K$ satisfies property $P(k)$ for some non-negative $k \in \mathbb{R}$, we are able to provide a lower bound for the Weil height of $j$ in terms of its discriminant $\Delta$. In order to make this assertion precise, we introduce some notation. For an elliptic curve $E$ defined over a number field $L$, denote by $h_F(E)$ its stable Faltings height \cite[pag. 354]{Faltings_1983} with Deligne's normalization \cite{Deligne_1985}. We continue writing $h: \overline{\mathbb{Q}} \to \mathbb{R}$ for the logarithmic Weil height of an algebraic number. 

\begin{proposition} \label{lower bound Weil height}
Let $j$ be a singular modulus of discriminant $\Delta=f^2 \Delta_K$, where $\Delta_K$ is the discriminant of the imaginary quadratic field $K$ relative to $j$. If for some $k\in \mathbb{R}_{\geq 0}$ property $P(k)$ holds for the non-principal real primitive Dirichlet character $\chi$ of discriminant $\Delta_K$, then
\[
h(j) \geq 1.509 \log |\Delta| + C
\]
for some effective constant $C=C(k) \in \mathbb{R}$.
\end{proposition}

\begin{proof}
Let $E/\mathbb{Q}(j)$ be an elliptic curve with $j(E)=j$. Using \cite[Lemma 7.9]{Gaudron_Remond_2014}, the logarithmic Weil height of $j$ can be bounded from below by the stable Faltings height of $E$ as follows
\begin{equation} \label{Gaudron_Remond inequality}
    h(j) \geq 12 h_F(E) + 8.64.
\end{equation}

We can explicitly compute the stable Faltings height of $E$ using the well-known results of Colmez \cite{Colmez_1998} and Nakkajima-Taguchi \cite{Nakkajima_Taguchi_1991}, as done for instance in \cite[Lemma 4.1]{Habegger_2010}. One has
\[
h_F(E)=\frac{1}{4} \log(|\Delta|) + \frac{1}{2}\frac{L'(\chi,1)}{L(\chi,1)} - \frac{1}{2} \left(\sum_{p\mid f} e_f(p) \log p \right) -\frac{1}{2} (\gamma+\log(2\pi))
\]

where $\gamma$ is the Euler-Mascheroni constant, $f$ is the conductor of the CM order and for a prime $p$ we define
$$e_f(p):=\frac{1-\chi(p)}{p-\chi(p)} \frac{1-p^{-v_p(f)}}{1-p^{-1}}.$$

Using property $P(k)$ we then get
\begin{align*}
    h_F(E) &\geq \frac{1}{4} \log(|\Delta|) + \frac{1}{2} \left(-0.2485 \log |\Delta_K| - k \right) - \frac{1}{2} \left(\sum_{p\mid f} e_f(p) \log p \right) -\frac{1}{2} (\gamma+\log(2\pi)) \\
    &= \frac{1}{4} \log(|\Delta|) + \frac{1}{2}(-0.2485 \log |\Delta| - 0.2485 \log f^{-2}-k) - \frac{1}{2} \left(\sum_{p\mid f} e_f(p) \log p \right) -\frac{1}{2} (\gamma+\log(2\pi)) \\
    &=0.12575 \log |\Delta| + 0.2485 \log f - \frac{1}{2} \left(\sum_{p\mid f} e_f(p) \log p \right) -\frac{1}{2} (\gamma+\log(2\pi) + k).
\end{align*}

We want to bound from below the quantity
\[
A(f):=0.2485 \log f - \frac{1}{2} \left(\sum_{p\mid f} e_f(p) \log p \right).
\]
To do this, one can proceed exactly as in \cite[Section 4]{BHK_2018}. First, one notices that
\[
e_f(p) \leq \frac{2}{p+1} \cdot \frac{1-p^{-v_p(f)}}{1-p^{-1}}
\]
by considering all the possible values of the Dirichlet character $\chi(p)$. Setting now for all $n\in \mathbb{N}_{>0}$ 
\[
\delta(n):=0.2485 \log n - \left(\sum_{p\mid n} \frac{\log p}{p+1} \cdot \frac{1-p^{-v_p(n)}}{1-p^{-1}} \right),
\]
one notices that $\delta(n)$ is an additive function and satisfies $\delta(p^{r+1}) \geq \delta(p^r)$ for all primes $p \in \mathbb{N}$ and integers $r>0$. Since one has $\delta(2), \delta(3) < 0$ and $\delta(p)>0$ for all primes $p\geq 5$, we deduce that $\delta(n) \geq \delta(2) + \delta(3)$ for all $n \in \mathbb{N}_{>0}$. We then have
\[
A(f) \geq \delta(f) \geq \delta(2) + \delta(3)= 0.2485(\log 2 + \log 3)- \left( \frac{\log 2}{3} + \frac{\log 3}{4} \right) > -0.0605.
\]
In conclusion, we obtain
\begin{equation} \label{Estimate on Faltings height}
    h_F(E) > 0.12575 \log |\Delta| - C_0
\end{equation}
where we set
\[
C_0=\frac{1}{2} (\gamma+\log(2\pi) + k) + 0.0605.
\]
Combining now \eqref{Gaudron_Remond inequality} with \eqref{Estimate on Faltings height} we obtain
\[
h(j)> 1.509 \log |\Delta| - 12C_0 + 8.64
\]
and this concludes the proof.
\end{proof}

We now state and prove a stronger version of Theorem \ref{main theorem j-0}, whose proof relies on the use of property $P(k)$ rather than on the use of GRH. We then show how Theorem \ref{main theorem j-0} follows from this stronger statement.

\begin{theorem} \label{true main theorem j-0}
Let $S_0$ be the set of rational primes congruent to $1$ modulo $3$, let $\ell \geq 5$ be an arbitrary prime and set $S_\ell:=S_0 \cup \{\ell \}$. Assume that all the Kronecker symbols attached to imaginary quadratic fields satisfy property $P(k)$ for some fixed $k \in \mathbb{R}_{\geq 0}$. Then there exists an effectively computable bound $B=B(\ell,k) \in \mathbb{R}_{\geq 0}$ such that the discriminant $\Delta_j$ of every singular $S_\ell$-unit $j \in \overline{\mathbb{Q}}$ satisfies $|\Delta_j| \leq B$. In particular, the set of singular moduli that are $S_\ell$-units is finite and its cardinality can be effectively bounded.
\end{theorem}

\begin{proof}
The proof is essentially identical to the proof of Theorem \ref{main theorem j-1728}, and we only sketch the argument. First of all, it is again sufficient to prove that, under the assumptions of the theorem, the set of singular $\{\ell\}$-units is finite and the discriminants of its elements can be effectively bounded. This follows in the same way as done at the beginning of the proof of Theorem \ref{main theorem j-1728}, but this time appealing to the proofs of \cite[Theorem 1.2 and Claim 3.1]{Campagna_2020}. Hence we are reduced to bounding the discriminants of singular $\{\ell\}$-units for $\ell \geq 5$ a prime congruent to $2$ modulo $3$. Let $j$ be a singular $\{\ell\}$-unit relative to the order $\mathcal{O}$ of discriminant $\Delta$. Again, one decomposes its logarithmic Weil height $h(j)$ into a sum $h(j)=A+N$ of an archimedean and a non-archimedean component.

The archimedean component $A$ has been studied in \cite[Corollary 3.2]{BHK_2018}. Here it is proved that, for $|\Delta| \geq 10^{14}$, we have
\begin{equation} \label{Archimedean part j-0}
A \leq \frac{12 F \log |\Delta|}{C_{\Delta}}+3 \log \frac{F |\Delta|^{1/2} \log |\Delta|}{C_{\Delta}}-3.77
\end{equation}
where $C_{\Delta}$ is the usual class number of the order of discriminant $\Delta$ and $F= \max \{2^{\omega(a)}: a\leq |\Delta|^{1/2} \}$. Note that, although \cite[Corollary 3.2]{BHK_2018} is only formulated for singular units, it also holds for general singular moduli (with the same proof) if one restricts to considering the archimedean component of their height. The non-archimedean part can be written as
\begin{align} \label{non archimedean part j-0}
N=\frac{\log \ell}{[\Q(j):\Q]} \sum_{\mathfrak{p} \mid \ell} v_{\mathfrak{p}}(j) f_{\mathfrak{p}}
\end{align} 

where $f_{\mathfrak{p}}$ denotes the residue degree of the prime $\mathfrak{p} \subseteq \Q(j)$ lying above $\ell$. Using Theorem \ref{Theorem valuation j-j0} with the elliptic curve ${E_0}_{/\mathbb{Q}}: y^2=x^3+1$ with $j(E_0)=0$ and noticing that $d_0=6$ because $\ell \geq 5$ we have
\[
v_{\mathfrak{p}}(j) \leq \max \left \{ 3\left( \frac{ \log (9 |\Delta|)}{2\log \ell} +\frac{1}{2} \right), 3 \right \}
\]
and, combining this estimate with equality \eqref{non archimedean part j-0}, we get
\begin{align} \label{non-Archimedean estimate j-0}
    N  \leq \max \left \{ \frac{3}{2}(\log (9|\Delta|) + \log \ell ), 3 \log \ell \right \}.
\end{align}

 A lower bound for the height $h(j)$ can be obtained by combining the conditional Proposition \ref{lower bound Weil height} with \eqref{The easy bound in BHK}. The conclusion of the proof can be then carried out in the same way as the proof of Theorem \ref{main theorem j-j0}.
\end{proof}

\begin{proof}[Proof of Theorem \ref{main theorem j-0}]
The fact that the Dirichlet $L$-functions attached to imaginary quadratic fields satisfy GRH implies in particular that for every non-principal real primitive Dirichlet character $\chi$ of negative discriminant $D$ we have
\[
\frac{L'(\chi, 1)}{L(\chi,1)}= O (\log \log |D|),
\]
where the implied constant is absolute, see for instance \cite[Section 3.1]{Granville_Stark_2000} or \cite[Theorems 1 and 3]{Ihara_2006} for the explicitness of the implied constant in the case $|D|>8$ (in the remaining cases, one can find an explicit bound for the absolute value of the logarithmic derivative for instance by first relating it to the Faltings height as done in \cite[Lemma 4.1]{Habegger_2010} and then by using some bounds on the difference between the Faltings height and the $j$-height \cite[Lemmas 2.6 and 3.2]{Pazuki_2019}). In particular, there exists $k \in \mathbb{R}_{\geq 0}$ such that property $P(k)$ holds for all Kronecker symbols attached to imaginary quadratic fields. Now one concludes by applying Theorem \ref{true main theorem j-0}.
\end{proof}

\section{An unsuccessful attempt at making Theorem \ref{main theorem j-0} unconditional} \label{sec: remove GRH?}

The goal of this section is to show that the naive attempt at making Theorem \ref{main theorem j-0} unconditional by improving the bounds obtained in Theorem \ref{Theorem valuation j-j0} is fruitless. Namely, we will prove that the order of magnitude of the bounds appearing in Theorem \ref{Theorem valuation j-j0} cannot be improved in general, at least in the case $j_0=0$. Under the condition that the considered prime $\ell$ divides the discriminant of the order $\mathcal{O}_j$ corresponding to the singular modulus $j$, it is easy to provide examples in which the second upper-bound of \eqref{valuation of j-j0} is reached. For instance, each of the singular moduli $j$ of discriminant $\Delta=-7\cdot 5^2$ is divided by the unique prime $\mathfrak{p}_5 \subseteq \mathbb{Q}(j)$ above $5$ and we have $v_{\mathfrak{p}_5}(j)=3$ (note that $d_0=6$ in this case). On the other hand, if $\ell$ does not divide the discriminant of $\mathcal{O}_j$ the claimed optimality follows from the following theorem.

\begin{theorem} \label{optimality of the bounds j-0}
Let $\ell \geq 5$ be a prime with $\ell \equiv 2 \text{ mod } 3$. There exists an infinite family of singular moduli $j$ whose corresponding discriminant $\Delta_j$ is coprime with $\ell$ and which satisfy
\[
v_\mu (j) \geq 3 \left(\frac{\log(|\Delta_j|-3)}{2 \log \ell} + \frac{1}{2} - \frac{\log 2}{\log \ell }\right)
\]
for some prime ideal $\mu \subseteq H_\mathcal{O}$ lying above $\ell$. Here $H_\mathcal{O}$ denotes the ring class field relative to the order $\mathcal{O}$ associated to $j$. 
\end{theorem}

To prove the theorem, we need two preliminary results.

\begin{proposition} \label{quaternionic order for j0=0}
Let $\ell \geq 5$ be a prime with  $\ell \equiv 2 \text{ mod } 3$ and consider the elliptic curve $E_0:y^2=x^3+1$ defined over $\mathbb{F}_\ell$. Then we have that
\[
\End_{\overline{\mathbb{F}}_\ell}(E_0)=\Z+\Z\zeta_3+\Z \xi+\Z \eta
\]
\noindent
is isomorphic to a maximal order in the quaternion algebra $\mathbb{B}_{\ell, \infty}$. Here, if $\zeta \in \overline{\mathbb{F}}_\ell$ denotes a fixed primitive $3$-rd root of unity, the endomorphisms $\zeta_3, \varphi, \xi, \eta \in \End_{\overline{\mathbb{F}}_\ell}(E_0)$ are such that $\zeta_3: (x,y) \mapsto (\zeta x,y)$, $\varphi:(x,y) \mapsto (x^\ell, y^\ell)$, $3\xi= 2 + \zeta_3 + 2\varphi +\zeta_3 \varphi$ and $3\eta = -1+\zeta_3-\varphi-2\zeta_3 \varphi$.
\end{proposition}

\begin{proof}
This proposition is certainly well known, but the author has not been able to find a suitable reference. One could directly verify that the given order is a maximal order in $\mathbb{B}_{\ell, \infty}$ whose elements represent endomorphisms of the elliptic curve $E_0$. We outline a possible strategy leading to the computation of this endomorphism ring, kindly suggested to the author by John Voight.

Since $\ell \equiv 2 \text{ mod } 3$, the elliptic curve $E_0$ is supersingular and $\mathcal{O}_{E_0}:=\End_{\overline{\F}_\ell} (E_0)$ is a maximal order in the quaternion algebra $\mathbb{B}_{\ell,\infty}$. Throughout this proof, we will always identify $\mathcal{O}_{E_0}$ with its image in $\mathbb{B}_{\ell,\infty}$ under some fixed embedding. Notice that $\mathcal{O}_{E_0}$ contains the subring $\mathcal{O}:= \mathbb{Z}[\zeta_3,\varphi]$. As we have $\varphi^2=-\ell$, the ring $\mathcal{O}$ is actually a rank $4$ suborder of $\mathcal{O}_{E_0}$ having $\mathbb{Z}$-basis $\{1,\zeta_3, \varphi, \zeta_3 \varphi \}$, and a discriminant computation shows that $|\mathcal{O}_{E_0} : \mathcal{O}|=3$. Hence, $\mathcal{O}_{E_0}$ contains an element of the form
\[
\alpha=\frac{A+B\zeta_3+C\varphi+D\zeta_3 \varphi}{3}, \hspace{0.5cm} A,B,C,D \in \Z
\]

with $3 \nmid \gcd(A,B,C,D)$. Since $\alpha$ is an element of a quaternion order, it is in particular integral. This implies that its reduced trace and norm must both be integers. One has
\begin{align*}
&\trd(\alpha)=\frac{2A-B}{3} \\
&\nrd(\alpha)= \frac{-\ell CD-AB+A^2+B^2+\ell (C^2+D^2)}{9},
\end{align*}
where $\trd(\cdot)$ and $\nrd(\cdot)$ denote respectively the reduced trace and the reduced norm in the quaternion algebra $\mathbb{B}_{\ell, \infty}$. Note now that, since $\Ogotic \subseteq \Ogotic_{E_0}$, the integers $A,B,C,D$ can be chosen to lie in $\{0,1,2\}$. Hence there is just a finite number of possibilities to check. A computation shows that the possible options for the tuple $(A,B,C,D)$ are the following four:
$$(1,2,1,2) \hspace{0.5cm} (1,2,2,1) \hspace{0.5cm} (2,1,1,2) \hspace{0.5cm} (2,1,2,1).$$
By adding the corresponding $\alpha$'s to the order $\Ogotic$ we get the following possibilities:
\begin{align*}
&(1,2,1,2), \hspace{0.5cm} \Ogotic_1: \Z+\Z \zeta_3+\Z\left(\frac{1}{3}-\frac{1}{3} \zeta_3 + \frac{1}{3}\varphi+\frac{2}{3}\zeta_3 \varphi\right) + \Z \left(-\frac{2}{3}-\frac{1}{3} \zeta_3 - \frac{2}{3}\varphi-\frac{1}{3}\zeta_3 \varphi\right) \\
&(1,2,2,1), \hspace{0.5cm} \Ogotic_2: \Z+\Z \zeta_3+\Z\left(\frac{1}{3}-\frac{1}{3} \zeta_3 + \frac{2}{3}\varphi+\frac{1}{3}\zeta_3 \varphi\right) + \Z \left(-\frac{2}{3}-\frac{1}{3} \zeta_3 - \frac{1}{3}\varphi-\frac{2}{3}\zeta_3 \varphi\right) \\
&(2,1,1,2), \hspace{0.5cm} \Ogotic_3: \Z+\Z \zeta_3+\Z\left(\frac{2}{3}+\frac{1}{3} \zeta_3 + \frac{1}{3}\varphi+\frac{2}{3}\zeta_3 \varphi\right) + \Z \left(-\frac{1}{3}+\frac{1}{3} \zeta_3 - \frac{2}{3}\varphi-\frac{1}{3}\zeta_3 \varphi\right) \\
&(2,1,2,1), \hspace{0.5cm} \Ogotic_4: \Z+\Z \zeta_3+\Z\left(\frac{2}{3}+\frac{1}{3} \zeta_3 + \frac{2}{3}\varphi+\frac{1}{3}\zeta_3 \varphi\right) + \Z \left(-\frac{1}{3}+\frac{1}{3} \zeta_3 - \frac{1}{3}\varphi-\frac{2}{3}\zeta_3 \varphi\right). \\
\end{align*}

Looking at the generators of these orders, we see that $\Ogotic_1=\Ogotic_4$ and $\Ogotic_2=\Ogotic_3$, so we discard the first two and we only consider $\Ogotic_3$ and $\Ogotic_4$. We need to decide which of these two rings is the "correct one". Indeed, the desired order must be identified with the endomorphism ring of the elliptic curve $E_0$. An element of the form $\frac{1}{3} \beta$, with $\beta \in \End_{\overline{\mathbb{F}}_\ell}(E_0)$, is an endomorphism of $E_0$ if and only if the endomorphism $\beta$ factors through the multiplication-by-3 morphism. This happens if and only if the $3$-torsion points of $E_0$ are in the kernel of $\beta$. The idea is then to compute the generators of the group of $3$-torsion points of $E_0$ and to test which order contains the "right" elements. The $3$-division polynomial of $E_0$ is
\[
\Phi_3(x)=3x(x^3+4),
\]
so we can choose as generators of the full $3$-torsion subgroup $E_0[3](\overline{\mathbb{F}}_\ell)$ the points
\[
P=(0,1), \hspace{1cm} Q=(-\sqrt[3]{4}, \sqrt{-3})
\]
for fixed choices of $\sqrt[3]{4}, \sqrt{-3} \in \overline{\mathbb{F}}_\ell$ as follows. Observe that for a prime $\ell \geq 5$ and $\ell \equiv 2 \text{ mod } 3$, all elements in $\F_\ell$ are cubes and $-3$ is not a square modulo $\ell$.
In view of this remark, we choose $Q$ in such a way that the first coordinate lies in $\mathbb{F}_\ell$. The second coordinate of $Q$ defines in any case a quadratic extension of $\mathbb{F}_\ell$, so that
\[
(\sqrt{-3})^\ell=-\sqrt{-3}.
\]
We are ready to verify that $\Ogotic_4$ is the correct order. Let 
\begin{align*}
&\Phi=2+\zeta_3+2\varphi+\zeta_3 \varphi \in \mathcal{O}_{E_0}\\
& \Psi=-1+\zeta_3-\varphi-2\zeta_3 \varphi \in \mathcal{O}_{E_0}.
\end{align*}
Then, using the fact that $2P=-P$ and $2Q=-Q$ we get that $\Phi=\Psi$ on the $3$-torsion points, so
\begin{align*}
&\Phi (P)=[2](0,1)+(0,1)+[2](0,1)+(0,1)=0 \\
&\Phi (Q)=(-\sqrt[3]{4},-\sqrt{-3})+(-\zeta \sqrt[3]{4}, \sqrt{-3})+[2]((-\sqrt[3]{4})^\ell, (\sqrt{-3})^\ell)+(\zeta (-\sqrt[3]{4})^\ell,(\sqrt{-3})^\ell)=0 
\end{align*}
which shows that $E_0[3] \subseteq \ker \Phi$ and $E_0[3] \subseteq \ker \Psi$. One can also verify that 
$$(2+\zeta_3+\varphi+2\zeta_3 \varphi)(Q) \neq 0.$$
This proves the proposition.
\end{proof}

\begin{proposition} \label{lifting}
Let $\mathcal{O}$ be an order in an imaginary quadratic field $K$ and $\ell \in \mathbb{N}$ be a prime inert in $K$ that does not divide the conductor of $\mathcal{O}$. Let $W$ be the ring of integers in the completion $\widehat{\mathbb{Q}^{\text{unr}}_\ell}$ of the maximal unramified extension of $\mathbb{Q}_\ell$, with uniformizer $\pi \in W$. Fix $n \in \mathbb{Z}_{>0}$ and let $E_0 \to \Spec(W/\pi^n)$ be an elliptic scheme such that the reduction modulo $\pi$ is supersingular. If $f: \mathcal{O} \hookrightarrow \End_{W/\pi^n}(E_0)$ is an optimal embedding, then there exists an elliptic curve $E_{/W}$ such that
\begin{itemize}
    \item $E \text{ mod } \pi^n \cong E_0$;
    \item $\End_W(E) \cong \mathcal{O}$.
\end{itemize}
\end{proposition}

\begin{proof}
This is an application of Gross and Zagier's generalization \cite[Proposition 2.7]{Gross_Zagier_1985} of the Deuring lifting Theorem \cite[Theorem 13.14]{Lang_1987}. Note that the proof of Gross and Zagier's result in the supersingular case does not require, in their notation, the ring $\mathbb{Z}[\alpha_0]$ to be integrally closed but only $\ell$ not dividing its conductor.

Write $\mathcal{O}=\mathbb{Z}[\tau]$ for some imaginary quadratic $\tau \in K$ and let $\alpha_0:=f(\tau)$. The endomorphism $\alpha_0$ induces on the tangent space $\text{Lie}(E_0)$ the multiplication by an element $w_0 \in W/\pi^n$ which is a root of the minimal polynomial $g(x)=x^2+Ax+B \in \mathbb{Z}[x]$ of $\tau$ over $\mathbb{Q}$. In order to apply \cite[Proposition 2.7]{Gross_Zagier_1985}, we need to show that there exists $w \in W$ such that $g(w)=0$ and $w \mod \pi^n=w_0$. Let $\beta:= w_0 \text{ mod } \pi \in \overline{\mathbb{F}}_\ell$. Then $\beta$ is a root of $g(x) \text{ mod } \pi$ lying in $\overline{\mathbb{F}}_\ell$. If $g'(\beta)=0$, then $\beta$ would actually lie in $\mathbb{F}_\ell$. However, since $\ell$ is inert in $K$ and does not divide the conductor of $\mathcal{O}$, the polynomial $g(x)$ is irreducible over $\mathbb{F}_\ell$ by the Kummer-Dedekind Theorem \cite[Proposition I.8.3]{Neukirch_book_1999}, and this implies that the derivative of $g(x)$ does not vanish on $\beta$ (an irreducible polynomial over a finite field has never a common zero with its derivative).
Then by Hensel's lemma there exists a unique $w \in W$ lifting $\beta$. This $w$ satisfies $g(w)=0$ and $w \text{ mod } \pi^n = w_0$ by construction.

We now apply \cite[Proposition 2.7]{Gross_Zagier_1985} to deduce that there exists an elliptic curve $E_{/W}$ and an endomorphism $\alpha \in \End_W(E)$ such that $E \text{ mod } \pi^n \cong E_0$ and $\alpha \text{ mod } \pi^n = \alpha_0$. In principle, the ring $\End_W(E)$ could strictly contain the order $\mathbb{Z}[\alpha]$. However, the reduction map identifies $\mathbb{Z}[\alpha]$ with $\mathcal{O}$, and the latter optimally embeds in $\End_{W/\pi^n}(E_0)$. Since the reduction map also embeds $\End_W(E) \hookrightarrow \End_{W/\pi^n}(E_0)$, we deduce that $\End_W(E)= \mathbb{Z}[\alpha] \cong \mathcal{O}$, as wanted. 
\end{proof}

\begin{proof}[Proof of Theorem \ref{optimality of the bounds j-0}]
Let $W$ be the ring of integers in the completion $\widehat{\mathbb{Q}^{\text{unr}}_\ell}$ of the maximal unramified extension of $\mathbb{Q}_\ell$, with uniformizer $\pi \in W$. For every $n \in \mathbb{N}$ let $R_n:=\End_{W/\pi^{n+1}}(E_0)$ be the endomorphism ring of the reduction of $E_0:y^2=x^3+1$ modulo $\pi^{n+1}$. By Theorem \ref{deformation} (a) we know that $R_n \cong \mathbb{Z}[\zeta_3]+\ell^n R_0$, where $R_0$ is the order appearing in the statement of Proposition \ref{quaternionic order for j0=0}.
A computation similar to the one carried out during the proof of Theorem \ref{Theorem valuation j-j0} shows that the ternary quadratic form induced by the reduced norm on the Gross lattice of $R_n$ with basis $\{1+2\zeta_3, \ell^n(2\xi -1), 2\ell^n (\varphi + \zeta_3 \varphi) \}$ is given by
\begin{equation} \label{Quadratic form j0=0}
    Q_{\ell,n}(X,Y,Z)=3X^2+\ell^{2n} \frac{4\ell+1}{3} Y^2 + 4\ell^{2n+1} Z^2+2\ell^n XY + 4\ell^{2n+1} YZ \in \mathbb{Z}[X,Y,Z]
\end{equation}
for all $n \in \mathbb{N}$. Proposition \ref{lifting} combined with Lemma \ref{Gross Lattice} implies in particular that, for any primitive triple of integers $(x,y,z) \in \mathbb{Z}^3$ such that $-D:=Q_{\ell,n}(x,y,z)$ is not divisible by $\ell$, there exists an elliptic curve $E_{/W}$ with complex multiplication by the order of discriminant $D$ and which is isomorphic to $E_0: y^2=x^3+1$ modulo $\pi^{n+1}$.  The primitive triple $(1,0,1)$ gives
\[
Q_{\ell,n}(1,0,1)=3+4\ell^{2n+1}
\]
which is not divisible by $\ell$. The j-invariant of the corresponding elliptic curve $E$ with CM by the order of discriminant $D$ will satisfy, by \cite[Proposition 2.3]{Gross_Zagier_1985}, the inequality
\[
v_\mu(j) \geq 3(n+1) = 3 \left(\frac{\log(|D|-3)}{2 \log \ell} + \frac{1}{2} - \frac{\log 2}{\log \ell }\right)
\]

for some prime $\mu \subseteq H_\mathcal{O}$ lying above $\ell$. This concludes the proof of the theorem.
\end{proof}

\section{A uniformity conjecture for singular moduli} \label{sec:uniformity conjecture}

In this final section we make some speculations, based on computer-assisted numerical calculations, concerning differences of singular moduli that are $S$-units. The starting point of our discussion is the following observation, which was already made in a previous version of this manuscript (compare also with \cite[Question 1.2]{Herrero_Menares_Rivera_Part3}): numerical computations seem to show that $j_{-11}=-2^{15}$, which is the unique singular modulus relative to the order of discriminant $\Delta=-11$, may also be the only singular modulus that is an $\{\ell\}$-unit for some prime $\ell$. In other words, it seems that the set $\mathcal{J}_1$ of singular moduli that are $S$-units for some set of primes $S$ of cardinality $1$ contains only one element, namely $j_{-11}$. It appears then natural to ask what happens if we increase the cardinality of the set $S$. Motivated by this question, we have performed some computations, whose results are displayed in Table \ref{Table singular S-units}. Let us describe the notation and the content of this table.

\begin{table}[b]
    \centering
    \begin{tabular}{cccc}
    \toprule
    $s$ & $\# \mathcal{J}_{s}(50000)$ & $\Delta_{\text{max},s}(50000)$ & primes appearing in the factorizations \\
    \midrule 
    $1$ & $1$ & $-11$ & $2$ \\
    \midrule
    $2$ & $9$ & $-83$ & 2, 3, 5, 11 \\
    \midrule
    $3$ & $28$ & $-227$ & $2,3,5,11,17,23,29,41 $ \\
    \midrule
    $4$ & $67$ & $-523$ & $2, 3, 5, 11, 17, 23, 29, 41, 47, 53, 59, 71, 83, 89 $ \\
    \midrule
    \multirow{2}{*}{$5$} & \multirow{2}{*}{$119$} & \multirow{2}{*}{$-987$} & $2, 3, 5, 7, 11, 17, 23, 29, 41, 47, 53, 59, 71, 83, 89,$  \\
     & & & $101, 107, 113, 131, 137, 149, 167, 173, 179, 281, 317 $\\
     \midrule
     \multirow{4}{*}{$6$} & \multirow{4}{*}{$195$} & \multirow{4}{*}{$-2043$} & $2, 3, 5, 7, 11, 13, 17, 23, 29, 41, 47, 53, 59, 71, 83, 89, $  \\
     & & & $101, 107, 113, 131, 137, 149, 167, 173, 179, 191, 197, $\\
     & & & $227, 233, 239, 251, 257, 263, 269, 281, 293, $\\
     & & & $ 311, 317, 353, 383 $ \\
     \midrule
     \multirow{5}{*}{$7$} & \multirow{5}{*}{$291$} & \multirow{5}{*}{$-2587$} & $2, 3, 5, 7, 11, 13, 17, 23, 29, 41, 47, 53, 59, 71, 83, 89, $  \\
     & & & $101, 107, 113, 131, 137, 149, 167, 173, 179, 191, 197,  $\\
     & & & $227, 233, 239, 251, 257, 263, 269, 281, 293, 311, 317, $\\
      & & & $ 347, 353, 359, 383, 389, 419, 431, 449, 467, 491, $\\
      & & & $  509, 521, 557, 569, 617, 641, 653, 677$ \\
\bottomrule \\
    \end{tabular}
    \caption{The table displays for $s\in \{1,...,7\}$ the number of imaginary quadratic discriminants up to $-5\cdot 10^4$ for which the corresponding singular moduli are $S$-units with $\#S=s$ (second column). The third column shows the biggest among the found discriminants and the fourth column shows all the primes appearing in the factorizations of the norms of the corresponding singular moduli.}
    \label{Table singular S-units}
\end{table}

If a singular modulus of discriminant $\Delta$ is an $S$-unit for some set $S$ of rational primes, then actually all singular moduli of discriminant $\Delta$ are singular $S$-units since, as we discussed in \cref{sec: prelude}, the set of singular moduli relative to the same discriminant form a full Galois orbit over $\mathbb{Q}$. For every $s,A \in \mathbb{N}$ denote then by $\mathcal{J}_s$ the set of Galois orbits of singular moduli that are $S$-units for some set $S$ of rational primes satisfying $\#S=s$ and by $\mathcal{J}_{s}(A)$ the subset of $\mathcal{J}_s$ consisting of those orbits whose corresponding singular moduli have discriminant $\Delta$ satisfying $|\Delta| \leq A$. 
Similarly, denote by $\Delta_{\text{max},s}$ (resp. $\Delta_{\text{max},s}(A)$) the biggest (in absolute value) imaginary quadratic discriminant such that there exists a singular modulus of discriminant $\Delta_{\text{max},s}$ whose Galois orbit belongs to $\mathcal{J}_s$ (resp. to $\mathcal{J}_{s}(A)$). If $\mathcal{J}_s$ is an infinite set, we put $\Delta_{\text{max},s}=-\infty$. Clearly, for every pair of natural numbers $A_1 \leq A_2$ we have
\[
    |\Delta_{\text{max},s}(A_1)| \leq |\Delta_{\text{max},s}(A_2)| \leq |\Delta_{\text{max},s}|.
\]
In Table \ref{Table singular S-units} we have computed, with the help of SAGE \cite{sagemath}, the cardinality of $\mathcal{J}_{s}(50000)$ for $s \in \{1,...,7\}$, and the corresponding $\Delta_{\text{max},s}(50000)$. Moreover, in the last column we have collected all the primes appearing in the norm factorizations of $j \in \mathcal{J}_{s}(50000)$.

The results displayed in Table \ref{Table singular S-units} show, for small values of $s \in \mathbb{N}$, that $\Delta_{\text{max},s}(50000)$ is much smaller compared to the bound $|\Delta| \leq 50000$ up to which we have performed our computations. For instance, we see that among all the Galois orbits of singular moduli with discriminant $|\Delta| \leq 50000$, only $9$ orbits contain singular $S$-units for some set $S$ with $\#S \leq 2$. Moreover, the biggest discriminant associated to a singular modulus belonging to one of these $9$ orbits is $\Delta=-83$. All this seems to suggest that $\Delta_{\text{max},s}(A)$ will remain constant for all $A \geq 50000$ \textit{i.e.} that $\Delta_{\text{max},s}(50000)=\Delta_{\text{max},s}$ for $s \in \{1,...,7\}$, which would mean that the number of primes dividing the norm of a singular modulus must increase as the absolute value of its discriminant gets bigger. If this were actually true, then the last column of Table \ref{Table singular S-units} would show which primes a set $S$ of cardinality $s$ must contain in order for the set of singular $S$-units whose norm has exactly $s$ prime factors to be non-empty (but for some of the resulting $s$-tuples the corresponding set of singular $S$-units is empty). For example, it seems from these computations that the set of singular $\{17,23\}$-units does not contain any singular modulus. All this discussion leads to the formulation of the following conjecture.

\begin{conjecture}[Uniformity conjecture for singular $S$-units] \label{uniformity conjecture}
For every $s \in \mathbb{N}$, the set $\mathcal{J}_s$ is finite.
\end{conjecture}

We could have equivalently formulated the above conjecture by saying that for every finite set $S$ of rational primes, the set of singular $S$-units is finite and its cardinality can be bounded only in terms of the cardinality of $S$, regardless from the primes contained in the latter set. The fact that this statement is equivalent to Conjecture \ref{uniformity conjecture} can be seen as follows: suppose that for every $s \in \mathbb{N}$ there exists a constant $C(s) \geq 0$ such that the set of singular $S$-units has cardinality bounded by $C(s)$ whenever $S$ is a set of rational primes satisfying $\#S=s$. Since being an $S$-unit is Galois invariant, this implies that $C(s)$ also bounds the size of the Galois orbit of any such singular $S$-unit, hence the class number of the corresponding imaginary quadratic order. By the Brauer-Siegel Theorem \cite[Chapter XIII, Theorem 4]{Lang_book_NT_1994} this entails a bound on the discriminant of any singular $S$-unit with $\#S=s$. Hence any such singular modulus must lie in a finite set that depends only on $s$, but not on $S$ and Conjecture \ref{uniformity conjecture} follows.

Inspecting the computations displayed in Table \ref{Table singular S-units}, one could also try to be more precise on the cardinality of the sets $\mathcal{J}_s$. For instance, we may ask the following

\begin{question}
Is it true that there exists only $1$ singular modulus which is an $S$-unit for $\#S=1$, and $9$ Galois orbits of singular moduli that are $S$-units for $\#S=2$?
\end{question}

The author finds it more difficult to formulate precise conjectures on how the number of primes dividing the norm of a singular modulus increases with respect to its discriminant. 

 Of course, there is no reason to restrict our attention to singular $S$-units. One can make similar conjectures for differences of the form $j-j_0$ with $j_0$ a fixed singular modulus. For instance, Table \ref{Table j-1728} shows how the above considerations seem to hold true also for differences of the form $j-1728$. The notation is the same used for Table 1, but with the necessary modifications: $\mathcal{J}_s$ is the set of Galois orbits of singular moduli $j$ such that $j-1728$ is an $S$-unit for some set $S$ of rational primes satisfying $\#S=s$, etc. Further computations with other differences $j-j_0$ would probably shed more light on whether it is possible that for every $s\in \mathbb{N}$, there is only a finite number of singular differences $j_1-j_2$ that are $S$-units for some sets $S$ of cardinality $s$. But we do not want to enter this territory here.

\begin{table}[h]
    \centering
    \begin{tabular}{cccc}
    \toprule
    $s$ & $\# \mathcal{J}_{s}(50000)$ & $\Delta_{\text{max},s}(50000)$ & primes appearing in the factorizations \\
    \midrule 
    $1$ & $0$ & $/$ & $/$ \\
    \midrule
    $2$ & $3$ & $-8$ & 2, 3, 7 \\
    \midrule
    $3$ & $14$ & $-52$ & $2, 3, 7, 11, 19, 23, 31, 43 $ \\
    \midrule
    \multirow{2}{*}{$4$} & \multirow{2}{*}{$31$} & \multirow{2}{*}{$-139$} & $2, 3, 7, 11, 19, 23, 31, 43, 47, 59, 67, 79, 83,  $ \\
     & & & $ 103, 127, 139 $ \\
    \midrule
    \multirow{2}{*}{$5$} & \multirow{2}{*}{$54$} & \multirow{2}{*}{$-259$} & $2, 3, 7, 11, 19, 23, 31, 43, 47, 59, 67, 71, 79, 83,$  \\
     & & & $103, 107, 127, 139, 151, 163, 211, 223 $\\
     \midrule
     \multirow{3}{*}{$6$} & \multirow{3}{*}{$93$} & \multirow{3}{*}{$-571$} & $2, 3, 5, 7, 11, 19, 23, 31, 43, 47, 59, 67, 71, 79, 83, $  \\
     & & & $ 103, 107, 127, 131, 139, 151, 163, 167, 179, 191, $\\
     & & & $ 199, 211, 223, 271, 283, 307, 331, 571 $\\
     \midrule
     \multirow{5}{*}{$7$} & \multirow{5}{*}{$145$} & \multirow{5}{*}{$-835$} & $ 2, 3, 5, 7, 11, 19, 23, 31, 43, 47, 59, 67, 71, 79, 83, $  \\
     & & & $103, 107, 127, 131, 139, 151, 163, 167, 179,  191, $\\
     & & & $ 199, 211, 223, 227, 239, 251, 271, 283, 307, 311,$\\
      & & & $   331, 367, 379, 383, 439, 463, 487, 499, 523, 547, $\\
      & & & $   571, 631, 691 $ \\
\bottomrule \\
    \end{tabular}
    \caption{The table displays for $s\in \{1,...,7\}$ the number of imaginary quadratic discriminants up to $-5\cdot 10^4$ for which the corresponding singular moduli $j$ are such that $j-1728$ is an $S$-unit for some $S$ with $\#S=s$ (second column). The third column shows the biggest among the found discriminants and the fourth column shows all the primes appearing in the factorizations of the corresponding norms of $j-1728$.}
    \label{Table j-1728}
\end{table}

\newpage

\newpage

\section*{Acknowledgments}
The author would like to thank Fabien Pazuki for his constant support and for the comments on previous versions of this manuscript. He would also like to thank Philipp Habegger and the University of Basel for their hospitality, and Jared Asuncion, Yuri Bilu, Gabriel Dill, Florent Jouve, Riccardo Pengo, Emanuele Tron for the useful discussions. Finally, he would like to thank the anonymous referee for their careful review work and the great number of suggestions that have very much improved the exposition of the previous version of the manuscript.

The author is grateful to Max Planck Institute for Mathematics in Bonn for its hospitality and financial support. The author is also supported by ANR-20-CE40-0003 Jinvariant.

\vspace{\baselineskip}
\noindent
\framebox[\textwidth]{
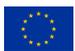
\begin{tabular*}{0.96\textwidth}{@{\extracolsep{\fill} }cp{0.84\textwidth}}
\raisebox{-0.7\height}{%
    \begin{tikzpicture}[y=0.80pt, x=0.8pt, yscale=-1, inner sep=0pt, outer sep=0pt, 
    scale=0.12]
    \definecolor{c003399}{RGB}{0,51,153}
    \definecolor{cffcc00}{RGB}{255,204,0}
    \begin{scope}[shift={(0,-872.36218)}]
      \path[shift={(0,872.36218)},fill=c003399,nonzero rule] (0.0000,0.0000) rectangle (270.0000,180.0000);
      \foreach \myshift in 
           {(0,812.36218), (0,932.36218), 
    		(60.0,872.36218), (-60.0,872.36218), 
    		(30.0,820.36218), (-30.0,820.36218),
    		(30.0,924.36218), (-30.0,924.36218),
    		(-52.0,842.36218), (52.0,842.36218), 
    		(52.0,902.36218), (-52.0,902.36218)}
        \path[shift=\myshift,fill=cffcc00,nonzero rule] (135.0000,80.0000) -- (137.2453,86.9096) -- (144.5106,86.9098) -- (138.6330,91.1804) -- (140.8778,98.0902) -- (135.0000,93.8200) -- (129.1222,98.0902) -- (131.3670,91.1804) -- (125.4894,86.9098) -- (132.7547,86.9096) -- cycle;
    \end{scope}
    \end{tikzpicture}%
}
&
This project has received funding from the European Union Horizon 2020 research and
innovation programme under the Marie Sk{\l}odowska-Curie grant agreement No 801199.
\end{tabular*}
}

\printbibliography

\end{document}